\documentclass[reqno]{amsart}

\usepackage{article-preamble}
\usepackage{comment}
\usepackage{outlines}
\usepackage{graphicx}
\hypersetup{colorlinks=true,
    linkcolor=blue,
    citecolor=blue,
    pdftitle={A study guide to "Kaufman and Falconer estimates for radial projections"},
    pdfauthor={Paige Bright, Ryan Bushling, Caleb Marshall, and Alex Ortiz}}

\newcommand{\ar}{\mathrm{r}}

\makeatletter

\renewcommand{\tocsection}[3]{%
  \indentlabel{\@ifnotempty{#2}{\bfseries\ignorespaces#1 #2\quad}}\bfseries#3}
\renewcommand{\tocsubsection}[3]{%
  \indentlabel{\@ifnotempty{#2}{\ignorespaces#1 #2\quad}}#3}

\newcommand\@dotsep{4.5}
\def\@tocline#1#2#3#4#5#6#7{\relax
  \ifnum #1>\c@tocdepth 
  \else
    \par \addpenalty\@secpenalty\addvspace{#2}%
    \begingroup \hyphenpenalty\@M
    \@ifempty{#4}{%
      \@tempdima\csname r@tocindent\number#1\endcsname\relax
    }{%
      \@tempdima#4\relax
    }%
    \parindent\z@ \leftskip#3\relax \advance\leftskip\@tempdima\relax
    \rightskip\@pnumwidth plus1em \parfillskip-\@pnumwidth
    #5\leavevmode\hskip-\@tempdima{#6}\nobreak
    \leaders\hbox{$\m@th\mkern \@dotsep mu\hbox{.}\mkern \@dotsep mu$}\hfill
    \nobreak
    \hbox to\@pnumwidth{\@tocpagenum{\ifnum#1=1\bfseries\fi#7}}\par
    \nobreak
    \endgroup
  \fi}
\AtBeginDocument{%
\expandafter\renewcommand\csname r@tocindent0\endcsname{0pt}
}
\def\l@subsection{\@tocline{2}{0pt}{2.5pc}{5pc}{}}

\makeatother

\begin{document}

\title[A study guide to ``Kaufman and Falconer estimates"]{A study guide to ``Kaufman and Falconer estimates \\ for radial projections"}
\author[P.~Bright]{Paige Bright}
\address{Department of Mathematics \\ Massachusetts Institute of Technology, 77 Massachusetts Avenue \\ Cambridge, MA, USA 02139-4307}
\email{paigeb@mit.edu}
\author[R.~Bushling]{Ryan Bushling}
\address{Department of Mathematics \\ University of Washington, Box 354350 \\ Seattle, WA, USA 98195-4350}
\email{reb28@uw.edu}
\author[C.~Marshall]{Caleb Marshall}
\address{Department of Mathematics \\ University of British Columbia, 1984 Mathematics Road \\ Vancouver, BC, Canada V6T 1Z2}
\email{cmarshall@math.ubc.ca}
\author[A.~Ortiz]{Alex Ortiz}
\address{Department of Mathematics \\ Massachusetts Institute of Technology, 77 Massachusetts Avenue \\ Cambridge, MA, USA 02139-4307}
\email{aeortiz@mit.edu}


\begin{abstract}
    This expository piece expounds on the major themes and clarifies technical details of the paper ``Kaufman and Falconer estimates for radial projections and a continuum version of Beck's theorem" of Orponen, Shmerkin, and Wang.
\end{abstract}

\maketitle

\vs{-0.15}
\parskip=0.1em
\tableofcontents \vs{-1.2}
\parskip=0.75em

\section{Introduction} \label{s:intro}

In 2022, Tuomas Orponen, Pablo Shmerkin, and Hong Wang wrote the article 
\begin{center}
    ``Kaufman and Falconer estimates for radial projections \\ and a continuum version of Beck's theorem"\footnote{This study guide is based on the \textit{arXiv} version of this paper (which is cited in this document). Recently, this paper has appeared in the \textit{Journal of Geometric and Functional Analysis} (2024): 1--38 \cite{orponen2024kaufman}.}
\end{center}
(henceforth referred to as ``OSW"). The authors applied $\epsilon$-improvements to the Furstenberg set problem and incidence estimates for balls and tubes in the plane to a bootstrapping scheme to prove lower bounds for radial projections that are sharp in some regimes. Here, we define \emph{radial projections} via the maps $\pi_x\colon \R^n \setminus \{x\} \to \mathbb{S}^{n-1}$ where 
\[
\pi_x(y) := \frac{y-x}{|y-x|}, \qquad y\in \R^n \setminus \{x\}.
\]

Their main results in the plane are the ``Kaufman-type'' theorem
\begin{thm}[\cite{orponen2022kaufman} Theorem 1.1] \label{thm:mainthm1.1}
    Let $X\subset \R^2$ be a (non-empty) Borel set which is not contained on any line. Then, for every Borel set $Y\subset \R^2$, 
    \[
    \sup_{x\in X} \dim \pi_x(Y\setminus \{x\}) \geq \min \{\dim X, \dim Y, 1\}.
    \]
\end{thm}
and the Falconer-type theorem
\begin{thm}
    [\cite{orponen2022kaufman} Theorem 1.2]\label{thm:mainthm1.2}
    Let $X,Y\subset \R^2$ be Borel sets with $X\neq \varnothing$ and $\dim Y>1$. Then, 
    \[
    \sup_{x\in X} \dim \pi_x(Y\setminus \{x\}) \geq \min \{\dim X+ \dim Y-1, 1\}.
    \]
\end{thm}
In the important special case where $X = Y$, Theorem \ref{thm:mainthm1.1} says that there must be a point in $X$ such that the radial projection $\pi_x(X\setminus\{x\})$ has dimension as large as possible, $\min\{\dim X, 1\}$.


In this study guide, we discuss the proofs of these two theorems in OSW's paper and hope to highlight the flexibility of the authors' device of ``thin tubes.'' In the process we will clarify certain (we think) key details in the arguments we found helpful while reading. We also describe some of the background on radial projections and the Furstenberg set problem as necessary to underline how small but meaningful improvements to this important problem contribute not only to small improvements in radial projections, but actually sharp radial projection theorems.

Finally, we will end by briefly exploring ongoing work in the area of radial projections.

\begin{rem}[Notation]
    Throughout we use the convention if $A, B$ are positive quantities, $A \lesssim B$ if $A \le CB$ for an absolute constant $C > 0$, and $A\sim B$ if both $A\lesssim B$ and $B\lesssim A$ hold.
\end{rem}

\subsection{Beck's theorem} \label{ss:becks theorem}

As a corollary of Theorem \ref{thm:mainthm1.1}, the authors obtained a continuum version of Beck's theorem from 1983, which states the following:

\begin{thm}[\cite{beck1983lattice} Beck's theorem]\label{becktheorem}
    Let $P \subset \R^2$ be a finite set of $N>1$ points in the plane. Then either 
    \begin{enumerate}[label={\normalfont \textbf{\arabic*.}},topsep=0pt]
        \item There is a line containing about $N$ many points, or
        \item $P$ spans about $N^2$ many lines.
    \end{enumerate}
\end{thm}

A reasonable continuum analog for Beck's theorem is that either \textbf{1.}~$X$ contains a lot of ``mass'' on some line $\ell$, or \textbf{2.}~the line set $\mathcal L(X)$ of lines spanned by $X$ has dimension as large as possible, namely $2$ if $\dim X > 1$ and $2\dim X$ otherwise.

\begin{rem}
    Note that $N$ points in $\R^2$ in generic position span $\binom{N}{2} \sim N^2$ many lines. In this sense, $2\dim X$ is the continuum analog of the bound $N^2$ from Beck's theorem.
\end{rem}

OSW are able to prove that the natural continuum analog of Beck's theorem is true. 

\begin{thm}[\cite{orponen2022kaufman} Corollary 1.3, Continuum version of Beck's theorem]
Let $X\subset \R^2$ be a Borel set. Then, either 
\begin{enumerate}[label={\normalfont \textbf{\arabic*.}}]
    \item there exists a line $\ell \subset \R^2$ such that $$\dim (X\setminus \ell) < \dim X, \qquad \text{or} $$
    \item the line set $\mathcal L(X)$ satisfies the inequality 
    \[
    \dim \mathcal L(X) \geq \min \{2\dim X, 2\}.
    \]
\end{enumerate}
\end{thm}

While the prospect of such a continuum analog actually \textit{motivated} the author's work in \cite{orponen2022kaufman}, in this paper it appears as a humble corollary of their main theorems on radial projections. From the point of view of the problem, it is natural to use radial projections to understand Beck's theorem, since for $X \subset \R^2$,
the set of points in $\pi_x(X\setminus \{x\})$ corresponds to the set of lines through $x\in X$.

Another natural thing to consider in the context of a continuum Beck's theorem is a version of the \emph{Furstenberg set problem} in the plane.

\begin{defn}\label{def:furst-lines}
    Let $\mathcal{A}(2,1)$ denote the affine Grassmannian of lines in $\R^2$. A subset $F \subset \mathcal{A}(2,1)$ is called an \define{$\bm{(s,t)}$-Furstenberg set of lines} (or a \emph{dual $(s,t)$-Furstenberg set}) if there exists a $(\ge t)$-dimensional set $X \subset \R^2$ of points such that for each $x \in X$, the collection $\mathcal{L}_x$ of lines in $F$ containing $x$ has dimension\footnote{See Section \ref{ss:Furstenberg Intro} for a bare-bones definition of the dimension of a set of lines, and the discussion following Definition \ref{def:delta-s-set} for the ``coordinate-free'' definition.} at least $s$. 
\end{defn}

The main goal regarding Furstenberg sets of lines is to establish a lower bound on the dimension of an $(s,t)$-Furstenberg set of lines which depends on $s$ and $t$ in an optimal way, and this is known as the \emph{(dual) Furstenberg set problem.}

By the correspondence sending the points of $\pi_x(X\setminus\{x\})$ to the collection $\mathcal L_x$ of lines passing through $x$, dimensional lower bounds for $(s,t)$-Furstenberg sets of lines could plausibly translate into lower bounds on the dimension of the line set $\mathcal L(X)$. In OSW's work, the authors manage to use Orponen and Shmerkin's $\epsilon$-improvement to the lower bound of an $(s,t)$-Furstenberg set of lines to establish strong lower bounds on radial projections in their main Theorem \ref{thm:mainthm1.1}. We discuss background on the Furstenberg set problem and the connection between the two different but equivalent forms of the problem in \S \ref{ss:Furstenberg Intro} below.

\subsubsection{Beck's theorem as motivation}

In incidence geometry, the Szemer\'edi--Trotter theorem gives a sharp bound for the number of incidences between finite sets of points and lines in the plane, and it also yields Beck's theorem as a corollary. As OSW note, the full strength of the Szemer\'edi--Trotter theorem is not required to prove Beck's theorem, and an appropriate $\epsilon$-improvement to simple lower bounds of the incidences suffices for the proof. In Appendix \ref{s:Beck}, we provide the details showing how this works for Beck's original theorem.


\subsection{Recent developments in radial projections} \label{ss:radial projections}

In this section, we provide some context behind the main theorems of OSW. Many researchers have been studying the effects of radial projections on Hausdorff dimension in recent years. The most recent thread of developments started in the context of vector spaces over finite fields, studied by Lund, Thang, and Huong Thu \cite{lund2022radial} and was motivated by work of Liu and Orponen \cite{liu2020hausdorff}. In particular, they showed the following:

\begin{thm}[\cite{lund2022radial} Theorem 1.1]\label{thm:lund-exceptional}
    Let $Y\subset \mathbb{F}_q^n$ and $M$ be a positive integer with $|Y|\gtrsim q^{n-1}$ and $M\lesssim q^{n-1}$. Then, 
    \[
    |\{x\in \mathbb{F}_q^n : |\pi_x(Y)| \leq M\}| \lesssim q^{n-1} M |Y|^{-1}.
    \]
\end{thm}

This lead to the following conjecture in $\R^n$ (when $k = n-1$), which was later resolved by the first author and Gan in \cite{bright2023exceptional}:
\begin{thm}[\cite{bright2023exceptional} Theorem 1] \label{thm:bright thm1}
    Let $Y\subset \R^n$ be a Borel set with $\dim Y\in (k,k+1]$ for some $k \in \{1,\dots, n-1\}$. Then, for $0<s<k$,
    \[
    \dim (\{x\in \R^n \setminus Y : \dim \pi_x(Y) < s\}) \leq \max \, \{k+s - \dim Y, 0\}.
    \]
\end{thm}

One can prove the following result (conjectured by Liu \cite{liu2020hausdorff}) as a consequence.

\begin{thm}[\cite{bright2023exceptional} Theorem 2]
\label{thm:bright thm2}
    Let $Y\subset \R^n$ be a Borel set with $\dim Y\in (k-1,k]$ for some $k \in \{1,\dots, n-1\}.$ Then we have 
    \[
    \dim (\{x\in \R^n\setminus Y : \dim \pi_x(Y) < \dim Y\})\leq k.
    \]
\end{thm}

In the context of radial projections, Theorems \ref{thm:lund-exceptional}, \ref{thm:bright thm1}, and \ref{thm:bright thm2} are known as \emph{exceptional set estimates} because they address the size of sets of points $x$ where the radial projection to $x$ signficantly compresses the dimension of $Y$---something that is ``exceptional'' from a measure-theoretic point of view. Theorems \ref{thm:bright thm1} and \ref{thm:bright thm2} are significant because they extend similar known bounds for radial projections in the plane to all Euclidean spaces in a unified way.

Although OSW's main radial projection theorems are results in the plane, by considering projections from higher dimensional Euclidean spaces to lower dimensional ones, OSW are able to prove an array of new results, including stronger exceptional set estimates than those of Theorems \ref{thm:bright thm1} and \ref{thm:bright thm2} in all dimensions.

In this study guide, we will focus on exploring the main Theorems \ref{thm:mainthm1.1} and \ref{thm:mainthm1.2} in the plane, although the device of thin tubes will be discussed in a natural way in all dimensions when we introduce it in Definition \ref{defn:thin-tubes-2}.

\subsection{Furstenberg sets}\label{ss:Furstenberg Intro}

An important blackbox for OSW's results is the $\epsilon$-improvement for the Furstenberg set problem. Before explaining what this result says, we give some background. Let $\mathcal{A}(2,1)$ denote the collection of all affine lines in the plane. We can regard $\mathcal A(2,1)$ as a metric space by parametrizing (non-vertical) lines in slope-intercept form $\ell(m,b) = \{(x,y)\in\R^2:y = mx+b\}$. The distance between two lines $\ell,\ell'$ can then be defined as $|m-m'| + |b-b'|$. In particular, a subset of non-vertical lines in $\mathcal A(2,1)$ has dimension $\alpha$ if the corresponding subset of $\R^2$ has dimension $\alpha$. This metric space structure on the non-vertical lines agrees with the ``coordinate-free'' definition which we discuss following Definition \ref{def:delta-s-set}.

\begin{defn}\label{def:furst-points}
    Let $0\leq s \leq 1$ and $0\leq t \leq 2$. An \define{$\bm{(s,t)}$-Furstenberg set of points} (often abbreviated to just an $(s,t)$-Furstenberg set) is a set $P\subset \R^2$ such that there exists a line set $\mathcal L\subset \mathcal{A}(2,1)$ with $\dim \mathcal L \geq t$ and $\dim (P\cap \ell) \geq s$ for all $\ell \in \mathcal L$. See \S\ref{ss:discrete-furstenberg} for the definition of the dimension of a line set.
\end{defn}

The \textit{Furstenberg set problem} asks what lower bounds one can obtain for the Hausdorff dimension of a Furstenberg set. 

A result due to Lutz--Stull and Wolff (for different ranges of $s$ and $t$, and later reproven by H\'era--Shmerkin--Yavicoli), proves that an $(s,t)$-Furstenberg set $P$ satisfies
\[
\dim P \geq s + \min\{s,t\};
\]
see \cite{lutz2020bounding}, \cite{wolff1999recent}, and \cite{hera2021improved}. In particular, if $t \ge 2$, the dimension of an $(s,t)$-Furstenberg set is at least $2s$. One of the striking features of OSW's work is that they are able to leverage just an $\epsilon$-improvment to the $2s$ lower bound, due Orponen and Shmerkin \cite{orponen2021hausdorff}, in order to prove the sharp continuum Beck's theorem. The following result (and its $\delta$-discretized version Theorem \ref{thm:furstenberg}) will be referred to throughout this study guide as ``the $\epsilon$-improved Furstenberg set lower bound.''
\begin{thm}[\cite{orponen2021hausdorff} Theorem 1.1]\label{thm:epsilon-improved-furst}
    For every $s\in (0,1)$ and $t\in (s,2]$, there exists $\epsilon = \epsilon(s,t)>0$ such that the following holds. Given $P$ is an $(s,t)$-Furstenberg set, then $$\dim P \geq 2s + \epsilon.$$
\end{thm}


\begin{rem}
    We note that in late 2023, Kevin Ren and Hong Wang resolved the Furstenberg set conjecture in the plane (see \cite{ren2023furstenberg}), namely that if $F$ is an $(s,t)$-Furstenberg set of points, then
    \[
    \dim F \ge \min\left\{s+t,\frac{3s+t}{2},s+1\right\}.
    \]
\end{rem}

As we mentioned in the discussion in Section \ref{ss:becks theorem}, in the context of Beck's theorem, Furstenberg sets of lines are natural objects to consider, but the discussion so far in this section has focused on Furstenberg sets of points and theorems justifying lower bounds on their dimensions. The relationship between the two is \emph{point-line duality}, which we briefly review here.

Given a point $(m,b)$ in the plane, we consider the associated \emph{dual line} $\mathbf D(m,b) = \{(x,mx+b): x\in \R\}$. Similarly, given a line in slope-intercept form $\ell=\{(x,c x + d):x\in\R\}$, we consider the \emph{dual point} $\mathbf D^\ast (\ell) = (-c,d) \in \R^2$. These maps are dual in the sense that they preserve incidences between points and (non-vertical) lines:
\begin{equation*}
p \in \ell \iff \mathbf D^\ast (\ell) \in \mathbf D(p).
\end{equation*}
If $X\subset \R^2$, we say $\mathbf D(X) = \bigcup_{x\in X}\mathbf D(x)$ is the corresponding dual set of lines, and if $\mathcal L\subset \mathcal A(2,1)$, $\mathbf D^*(\mathcal L) = \bigcup_{\ell\in\mathcal L}\mathbf D^*(\ell)$ is the corresponding dual set of points. Locally, both $\mathbf D,\mathbf D^\ast$ are bilipschitz mappings between the standard metric on $\R^2$ and the metric on $\mathcal A(2,1)$. Therefore, the Hausdorff dimension of a set of points agrees with the Hausdorff dimension of the corresponding dual set of lines, and vice-versa.

Given an $(s,t)$-Furstenberg set of lines, what does its image look like under $\mathbf D$ and $\mathbf D^*$? Recall the data of an $(s,t)$-Furstenberg set of lines (Definition \ref{def:furst-lines}):
\begin{itemize}
    \item A $t$-dimensional set of points $X\subset \R^2$, and
    \item For each $x\in X$, an $s$-dimensional set of lines $\mathcal L_x$ \emph{containing} $x$.
\end{itemize}
Under $\mathbf D$ and $\mathbf D^*$, the corresponding dual data is given by
\begin{itemize}
    \item A $t$-dimensional set of lines $\mathbf D(X)\subset\mathcal A(2,1)$, and
    \item For each line $\mathbf D(x)\in\mathbf D(X)$, an $s$-dimensional set of points $\mathbf D^*(\mathcal L_x)$ \emph{contained in} $\mathbf D(x)$.
\end{itemize}

Comparing with Definition \ref{def:furst-points}, an $(s,t)$-Furstenberg set of lines has a corresponding $(s,t)$-Furstenberg set of points, and vice-versa. Therefore, the question of finding optimal lower bounds for the dimension of a Furstenberg set of lines is equivalent to finding optimal lower bounds for the dimension of a Furstenberg set of points.

\begin{rem}
    With our notational conventions, an $(s,t)$-Furstenberg set of points precisely corresponds to an $(s,t)$-Furstenberg set of lines under duality, but be aware that the order of ``$s$'' and ``$t$'' in the definition matters, and in some conventions, the order of $(s,t)$ may be reversed on taking duality.
\end{rem}

\subsection{An outline of the study guide}

In this study guide, we focus on OSW's radial projection results\textemdash Theorems \ref{thm:mainthm1.1} and \ref{thm:mainthm1.2}\textemdash in \S\ref{s:Section 2} and \S\ref{s:Section 3}, respectively, in order to highlight how the $\epsilon$-improvement in the  Furstenberg set problem translates into a sharp theorem (continuum Beck). Then, in \S\ref{s:Section 4}, we discuss further results that have come out since the work of OSW, as well as work that is currently in progress. The topics include a continuum Erd\H{o}s--Beck theorem and generalizations of Theorems \ref{thm:mainthm1.1} and \ref{thm:mainthm1.2}.

\subsection{Notation and preliminaries}
We use the following notation for certain standard ideas coming from measure theory. 
\begin{itemize}

    \item If $\mu$ is a measure on $A$, we let $\spt \mu$ denote the support of $\mu$. This is the smallest closed subset $F \subset A$ such that $\mu (A \setminus F) = 0$.

    \smallskip
    
    \item We let $\mathcal P(X)$ denote the space of Borel probability measures $\mu$ with $\spt \mu \subset X$.

    \smallskip
    
    \item If $T\colon X\to Y$ is a measurable map, and $\mu$ is a measure on $X$, we let $T\mu$ denote the pushforward of $\mu$ by $T$. The pushforward defines a measure on $Y$ via the formula $(T\mu)(E) = \mu(T^{-1}(E))$.

    \smallskip
    
    \item A measure $\mu$ on a metric space $X$ is $\sigma$-Frostman if $\mu(B(x,r))\le C_\mu r^\sigma$ for all $x\in X$ and $r > 0$.
    
    \smallskip
    
    \item If $\mu$ is a measure, and $\sigma > 0$, we let $I_\sigma(\mu)$ denote the $\sigma$-energy of $\mu$, defined by $\iint |x-y|^{-\sigma}\,d\mu(x)\,d\mu(y)$.
    The notions of being $\sigma$-Frostman and having finite $\sigma$-energy are closely related. For instance,  if $\mu$ is $\sigma$-Frostman, then $I_\sigma(\mu)<\infty$. 
    
\end{itemize}
We found Chapter 2 of Mattila's book \cite{mattila2015fourier} to be a helpful reference for the commonly used facts from geometric measure theory in OSW. In this study guide, we also use the following notation.
\begin{itemize}
    \item If $X$ is a (separable) metric space and $A \subset X$, then $\mathcal{H}^s$ and $\dim A$ always refer to the $s$-dimensional Hausdorff measure and Hausdorff dimension on $X$. See Chapter 5 of \cite{mattila1995geometry} for a treatment of these objects in general metric spaces.

    \smallskip
    
    \item The map $\pi_x \colon \mathbb{R}^n \setminus \{x\} \rightarrow \mathbb{S}^{n-1}$ defined as $$\pi_x(y) = \frac{y-x}{|y-x|}, \quad y \in \mathbb{R}^n \setminus \{x\},$$
    always denotes radial projection onto the unit sphere about $x \in \mathbb{R}^n$.

    \smallskip

    \item For $e \in \mathbb{S}^{n-1}$, let $P_e \colon  \mathbb{R}^n \rightarrow \ell_e$ denote orthogonal projection onto the the line $\ell_e = \{t e : t \in \mathbb{R}\}$. The formula is
    $$
    P_e (x) : = (x \cdot e) e, \quad \forall x \in \mathbb{R}^n.
    $$

    \smallskip
    
    \item For integers $1 \leq m \leq n$, we let $\mathcal{G}(n,m)$ denote the Grassmannian of $m$-dimensional linear subspaces in $\mathbb{R}^n$; the set $\mathcal{A}(n,m)$ is the corresponding affine Grassmannian of $m$-dimensional affine subspaces in $\mathbb{R}^n$. See Chapter 4 of \cite{mattila1995geometry} for a discussion of these objects as metric spaces, and their associated invariant measures.

    \smallskip

    \item If $G$ is a subset of a product set $X \times Y$ we let $G|_x := \{ y \in Y \!: (x,y) \in G \}$ denote the $x$-slice of $G$, and $G|_y=\{x\in X:(x,y)\in G\}$ denotes the $y$-slice of $G$.

    \smallskip
    
    \item Let $\delta > 0$. If $X$ is a metric space and $A \subset X$, then $\vert A \vert_{\delta}$ denotes the $\delta$-covering number of $A$. That is, the (minimum) number of balls of diameter $\delta$ needed to cover $A$.
\end{itemize}

\bigskip
\begin{sloppypar}
\noindent {\bf Acknowledgements.} This study guide was developed as a part of the Study Guide Writing Workshop 2023 at the University of Pennsylvania. We would like to thank Josh Zahl for his mentorship, Hong Wang for her insightful discussions, and Phil Gressman, Yumeng Ou, Hong Wang, and Joshua Zahl for organizing this workshop.
\end{sloppypar}

\section{The Kaufman-type estimate, Theorem 1.1} \label{s:Section 2}

The first main result of OSW is the so-called ``Kaufman-type" radial projection theorem. (See Appendix \ref{Appendix C} for its relation to Kaufman's classical exceptional set estimate.)

\begin{thm}[\cite{orponen2022kaufman} Theorem 1.1] \label{thm:radial-kaufman}
    Let $X\subset \R^2$ be a (non-empty) Borel set which is not contained on any line. Then, for every Borel set $Y\subset \R^2$, 
    \begin{equation} \label{eq:radial-kaufman}
    \sup_{x\in X} \dim \pi_x(Y\setminus \{x\}) \geq \min \{\dim X, \dim Y, 1\}.
    \end{equation}
\end{thm}



The authors prove this by a bootstrapping scheme that outsources the base case to \cite{shmerkin2022non} and exploits the Furstenberg set bound in \cite{orponen2021hausdorff} to execute the actual bootstrapping. Roughly, the argument reads as follows.
\begin{enumerate}[label=\textbf{\arabic*.}, itemsep=3pt, topsep=0pt]
    \item As is generally the case in geometric measure theory, rephrase the problem concerning the geometry of sets as one concerning the geometry of measures. In particular, we replace the question about the radial projections of $Y$ about points of $X$ with the question of whether a certain pair $(\mu,\nu)$ of Frostman measures have \textit{thin tubes}. This notion quantifies how well one measure ``sees" another; see Definition \ref{defn:thin-tubes-2}. 
    \item Initiate the bootstrapping scheme with Theorem B.1 of \cite{shmerkin2022non}. This states precisely that the required pair $(\mu,\nu)$ of $s$-Frostman measures have $\beta$-thin tubes for some $\beta > 0$, so there is no work to be done here.
    \item Argue that, if a pair $(\mu,\nu)$ of $s$-Frostman measures have $\sigma$-thin tubes for some $\beta \leq \sigma < s$, then in fact there is some $\eta = \eta(s,\sigma) > 0$ such that the pair have $(\sigma+\eta)$-thin tubes. It is important here that $\eta$ not shrink too rapidly as $\sigma \uparrow s$, i.e., that the bootstrapping scheme maintain enough ``momentum" to carry $\sigma$ to the target value $s$.
    \item From Steps 2 and 3, conclude a simple criterion\textemdash a measure version of the hypotheses of Theorem \ref{thm:radial-kaufman}\textemdash for $(\mu,\nu)$ to have $\sigma$-thin tubes for all $\sigma < s$.
    \item By an appropriate choice of measures $\mu$ on $X$ and $\nu$ on $Y$, conclude from Step 1 that \eqref{eq:radial-kaufman} holds.
\end{enumerate}

Step 3 is by far the most arduous of the five, although developing intuition for thin tubes is not so trivial as Steps 1 and 5 might suggest.

\subsection{The geometry of thin tubes} \label{ss:thin-tubes}

The definition of ``thin tubes" first appeared in \cite{shmerkin2021distance} for much the same purpose as in OSW. On the most basic level, the definition should seem natural: in order to understand the properties of the pushforward measure $\pi_x \nu \in \mathcal{P}(\bbs^{n-1})$, one ought to examine the properties of $\nu$ on (neighborhoods of) the fibers of $\pi_x$. On the other hand, the particulars of the definition can seem unmotivated, so we begin with a provisional definition based on \cite{shmerkin2021distance}.

For the remainder of this section we work in the closed unit ball $B^n \subset \R^n$ (or $B^2 \subset \R^2$, when the arguments do not generalize), but any sufficiently large disc containing positive-measure subsets of $X$ and $Y$ would do just as well.

\begin{defn} \label{defn:thin-tubes-1}
    Let $x \in B^n$ and $\nu \in \mathcal{P}(B^n)$. We say that $(x,\nu)$ have \define{$\bm{(\sigma,K,c)}$-thin tubes} if there exists a Borel set $F \subset \spt \nu$ such that $\nu(F) \geq c$ and
    \begin{equation} \label{eq:thin-tubes-1}
        \nu(T \cap F) \+\leq\+ K\delta^\sigma \qquad \text{for all open } \delta \text{-tubes } T \text{ containing } x
    \end{equation}
    and for all $\delta \leq 1$. When the constants are unimportant, we simply say that $(x,\nu)$ have \define{$\bm{\sigma}$-thin tubes}.
\end{defn}

The parameter $c \in (0,1]$ is the mass of $\spt \nu$ with the desired geometric property, and our only concern should be that it not stray too far from $1$. As we bootstrap from $\sigma$-thin tubes to $(\sigma+\eta)$-thin tubes, the constant $c$ incurs a loss that must not be left uncontrolled. However, the constant $K$ is unimportant for the same reason that the analogous constant in the definition of Frostman measure (cf.~Appendix \ref{Appendix B}) is frequently unimportant.

\begin{lem} \label{lem:frostman-iff-tubes}
    Suppose $x$ and $\spt \nu$ are positively separated and $\sigma > 0$. Then the pushforward measure $\pi_x \nu$ is $\sigma$-Frostman if and only if $(x,\nu)$ have $(\sigma,K,1)$-thin tubes. Consequently, $\mathcal{H}^\sigma(\pi_x(Y)) > 0$ if and only if $Y$ supports a probability measure $\nu$ such that $(x,\nu)$ have $\sigma$-thin tubes.
\end{lem}

\textit{Proof.} The second statement follows immediately from the first by Frostman's lemma, as we can always take $c = 1$ by restricting and renormalizing the measure. To prove the first statement, we observe that, since $\dist(x,\spt \nu) > 0$ and $\spt \nu$ is compact, the $\nu$-mass of a $\delta$-tube containing $x$ is always comparable to the $\nu$-mass of a cone of pitch $\delta$ with vertex at $x$. Hence, \eqref{eq:thin-tubes-1} is equivalent to the statement that $\pi_x \nu(D) \lesssim \delta^\sigma$ for every disc $D \subset \bbs^{n-1}$ of diameter $\delta$ and center incident to the axis of $T$, i.e., that $\pi_x \nu$ is $\sigma$-Frostman. \hfill $\square$

\vs{0.3}

\begin{figure}[h!]
    \centering
    \includegraphics[width=.6\textwidth]{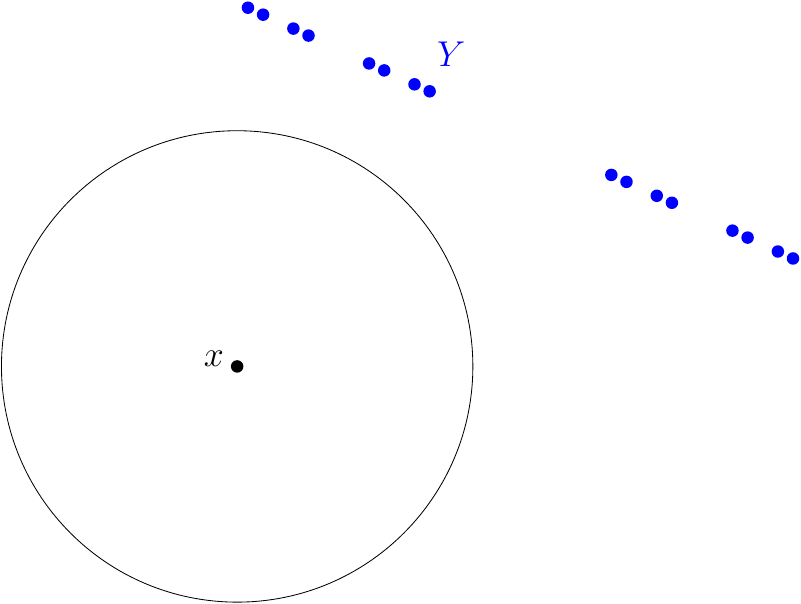}
    \caption{An illustration of a point $x$, and a measure supported on $Y$ such that $(x,\nu)$ has $\sigma$-thin tubes for some $\sigma>0$.}
    \label{fig:thin-tubes-1}
\end{figure}

With a sound intuition for what it means when a point and a measure have thin tubes, the generalization to pairs of measures having thin tubes should not be too daunting. Given a set $G$ contained in a product set $X \times Y$ and a point $x \in X$, denote by $G|_x := \{ y \in Y \!: (x,y) \in G \}$ the \textit{$x$-slice} of $G$.

\begin{defn} \label{defn:thin-tubes-2}
    Let $\mu,\nu \in \mathcal{P}(B^n)$. We say that $(\mu,\nu)$ have \define{$\bm{(\sigma,K,c)}$-thin tubes} if there exists a set $G \subset \spt \mu \times \spt \nu$ such that $(\mu \times \nu)(G) \geq c$ and, for each $x \in \spt \mu$,
    \begin{equation} \label{eq:thin-tubes-2}
        \nu(T \cap G|_x) \+\leq\+ K\delta^\sigma \qquad \text{for all open } \delta \text{-tubes } T \text{ containing } x
    \end{equation}
    and for all $\delta \leq 1$. We also say that any pair $(\tilde{\mu}, \tilde{\nu})$ of nonzero finite measures have $(\sigma,K,c)$-thin tubes when the normalized pair $\big( |\tilde{\mu}|^{-1} \tilde{\mu}, |\tilde{\nu}|^{-1} \tilde{\nu} \big)$ does.
\end{defn}

Note that Equations \eqref{eq:thin-tubes-2} and \eqref{eq:thin-tubes-1} are essentially identical: the former simply replaces $F$, which implicitly depends on $x$, with the $x$-slice $G|_x$. However, $(\mu,\nu)$ having $(t,K,c)$-thin tubes is weaker than $(x,\nu)$ having $(t,K,c)$-thin tubes for all $x \in \spt \mu$, as the definition does not constrain the $\nu$-measure of any \textit{individual} slice $G|_x$. We \textit{do} require the growth condition \eqref{eq:thin-tubes-1}-\eqref{eq:thin-tubes-2} to hold for all $x \in \spt \mu$, but the slices $G|_x$ only need to have $\nu$-measure $\geq c$ \textit{on average}. The notion of ``weak thin tubes" OSW mention in passing involves yet another relaxation of this sort, with \eqref{eq:thin-tubes-1}-\eqref{eq:thin-tubes-2} only being required for $x$ in a set of $\mu$-measure $\geq c$.

\begin{lem}\label{lem:KEY thin tubes}
    Suppose $\mu$ is a measure with $\spt\mu=X$ and $\nu$ is a measure with $\spt \nu = Y$. If $(\mu,\nu)$ have $(\sigma,K,c)$-thin tubes for some $\sigma,K,c > 0$, then
    \[
    \sup_{x\in X}\dim\pi_x(Y\setminus\{x\}) \ge \sigma.
    \]
\end{lem}
\begin{proof}
    This is essentially Remark 2.3 in OSW. For this result, the value of $K$ is not important: the only things that matter are $\sigma$ and $c>0$. By Fubini's theorem,
    \begin{align*}
        \int \nu(G|_x)\,d\mu(x) = (\mu\times \nu)(G) \ge c > 0.
    \end{align*}
    Hence, there exists $x\in X$ such that $\nu(G|_x)>0$. Here is where $\sigma > 0$ matters; in the definition of thin tubes, $x\in T$, so we have $\nu(G|_x\cap\{x\}) \le \nu(G|_x\cap T) \le K\cdot r^\sigma$. By taking $r$ arbitrarily small, we see that $\nu(G|_x\cap\{x\})=0$. Therefore, we may pick a compact subset $C\subset G|_x\setminus\{x\}\subset Y\setminus\{x\}$ such that $\nu(C)>0$. By the definition of thin tubes, we find that $\pi_x(\nu|_C)$ satisfies a $\sigma$-Frostman condition with the constant $K$ of $(\sigma,K,c)$-thin tubes. By Lemma \ref{lem:dim lower bound}, $\dim\pi_x(Y\setminus\{x\}) \ge \dim\pi_x(C) \ge \sigma$.
\end{proof}

\begin{figure}[h!]
    \centering
    \includegraphics[width=.55\textwidth]{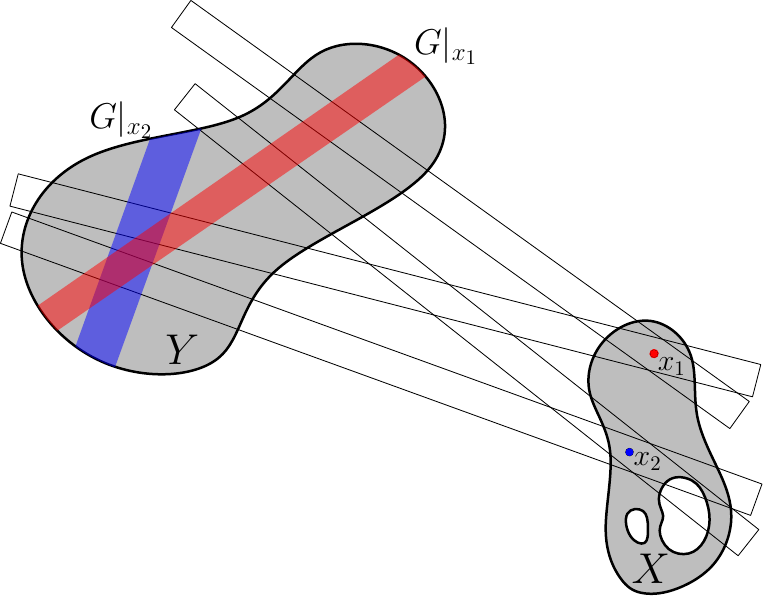}
    \caption{A schematic diagram of a pair $(\mu,\nu)$ of measures with $(\sigma,K,c)$-thin tubes. For each $x \in X$, we can discard (on average) no more than $(1-c) \nu(Y)$ of the mass of $Y$ and have $(x,\nu)$ with $(\sigma,K,c)$-thin tubes in the sense of Definition \ref{defn:thin-tubes-1}.}
    \label{fig:thin-tubes-2}
\end{figure}



With the mechanics of thin tubes established, we can now undertake Step 1 of our Theorem 1.1 proof scheme and rephrase the result in terms of thin tubes. \textit{Per se}, this does not play a role in the logic and only comes to bear in Step 5, but without this initial step the entirety of OSW \S 2 will have no apparent relation to the problem at hand until the final few sentences.

The claim is that, to prove Theorem \ref{thm:radial-kaufman}, it suffices to show that a pair $(\mu,\nu)$ of $s$-Frostman measures have $\sigma$-thin tubes for all $\sigma < s \leq 1$. To see this, let $X,Y \subset \R^2$ be as in the theorem statement and $s < \min \{ \dim X, \dim Y \}$. Then, by Frostman's lemma, there exist $s$-Frostman probability measures $\mu$ on $X$ and $\nu$ on $Y$. By restricting the supports to compact sets and renormalizing, we can assume without loss of generality that $\dist(\spt \mu, \spt \nu) > 0$. If $\spt \nu$ is contained in a line $\ell$, then, since $X \not\subset \ell$ by hypothesis, we can project $Y$ radially about any $x \in X \setminus \ell$ to conclude that
\begin{equation*}
    \dim \pi_x(Y \setminus \{x\}) \geq \dim \pi_x(\spt \nu) = \dim \spt \nu \geq s.
\end{equation*}
(This is because the restricted projection $\pi_x |_\ell \!: \ell \to \bbs^1$ is locally bi-Lipschitz onto its image.) If this is not the case, then we will see in Step 4 that $(\mu,\nu)$ have $\sigma$-thin tubes for all $\sigma < \min \{ s, 1 \}$ (cf.~Lemma \ref{lem:lem2.8}). By Lemma \ref{lem:KEY thin tubes}, it follows that
\begin{equation*}
    \sup_{x \in X} \dim \pi_x(Y \setminus \{x\}) \geq \min \{ \sigma, 1 \}
\end{equation*}
for all $\sigma < \min \{ s, 1 \}$, and taking the supremum over such $\sigma$ gives the desired result.

This argument is a minor variant of the one appearing at the end of OSW \S 2. We now circle back to the thick of the proof of Theorem 1.1.

\subsection{A careful walk-through of the bootstrapping argument} \label{ss:thm1.1-bootstrapping}

Let us discuss the bootstrapping argument in detail, emphasizing how Steps 2, 3, and 4 lead to Step 5. Recall that the goal is to show that, if $\mu$ and $\nu$ are $s$-Frostman, then, under mild conditions, $(\mu,\nu)$ have $\sigma$-thin tubes for all $\sigma < s$. We do this by working our way up to larger values of $\sigma$ from the following black boxed lemma.

\begin{lem}[\cite{orponen2022kaufman} Proposition 2.4, \cite{shmerkin2022non} Theorem B.1] \label{lem:prop2.4}
    For all $C, \delta, \epsilon, s > 0$, there exist
    \begin{equation*}
        \beta := \beta(s) \in (0,s), \quad \tau := \tau(\epsilon,s) > 0, \quad \text{and} \quad K := K(C, \delta, \epsilon, s) > 0
    \end{equation*}
    such that the following holds. Let $\mu,\nu \in \mathcal{P}(B^2)$ be $s$-Frostman measures with positively separated supports. If each $\delta$-tube in $\R^2$ has $\nu$-mass at most $\tau$, then $(\mu,\nu)$ have $(\beta,K,1-\epsilon)$-thin tubes.
\end{lem}

\begin{rem}
    The statement of the lemma is somewhat surprising. The enemy scenario to keep in mind is if $\mu$ and $\nu$ are both supported on a common line, then neither $(\mu,\nu)$ nor $(\nu,\mu)$ could have $\beta$-thin tubes for any $\beta > 0$, since radial projections of points on a line to a point on the same line is a set of at most two points. The hypothesis that $\nu$ does not give much mass to any $\delta$-tube rules out this pathological configuration, and the conclusion of Lemma \ref{lem:prop2.4} is that this is all we need to guarantee $(\mu,\nu)$ has $\beta$-thin tubes.
\end{rem}


The additional details provided in the statements in \cite{orponen2022kaufman} are not so important: uniform control of the parameters is necessary in \cite{orponen2022kaufman} Lemma 2.8, as that result is used iteratively in the ensuing corollaries, but Proposition 2.4 is only applied once, namely, as a base case (Step 2 in the proof scheme above).

\begin{lem}[\cite{orponen2022kaufman} Lemma 2.8, The ``key lemma"] \label{lem:lem2.8}
    For all $0 < \sigma < s \leq 1$, there exists $\eta > 0$ such that the following holds. Let $K,\epsilon > 0$ and let $\mu,\nu \in \mathcal{P}(B^2)$ be $s$-Frostman measures with positively separated supports. If $(\mu,\nu)$ have $(\sigma,K,1-\epsilon)$-thin tubes for some $\sigma \in [\beta,s)$, where $\beta$ is as in Lemma \ref{lem:prop2.4}, then in fact $(\mu,\nu)$ have $(\sigma+\eta,K',1-5\epsilon)$-thin tubes for some $K' > 0$. Furthermore, $\eta$ is bounded away from $0$ on compact subsets of $\{ (\sigma,s) \in (0,1]^2 \!: \sigma \in [\beta(s),s) \}$.
\end{lem}

The conclusion of Lemma \ref{lem:lem2.8} tells us that the property of ``having $(\sigma, K, c)$-thin tubes'' is \textit{stable} under increases in the exponent $\sigma$ in Equation \eqref{eq:thin-tubes-2}, by which we mean that one can increase $\sigma$ to $\sigma + \eta$ at the expense of increasing $K$ and decreasing $c$. Remember: the parameter $K \geq 1$ is some multiplicative constant, and $0 \leq c \leq 1$ is the proportion of the support of $\mu \times \nu$ where  the Frostman-type inequality \eqref{eq:thin-tubes-2} holds. See also Table \ref{tab:parameters}.

A detailed treatment of Lemma \ref{lem:lem2.8} follows below in \S \ref{ss:thm1.1-key-lemma}. For now, we take it at face value and use it to conclude the existence of $\sigma$-thin tubes for all $\sigma < s$. In OSW, this is done over the course of two corollaries and a lemma, but Corollary 2.18 (a ``simple criterion for thin tubes") and Lemma 2.22 can reasonably be tucked away under the more pivotal Corollary 2.21 (a ``simpler criterion for thin tubes," which in fact encompasses Lemma 2.22 as a special case). The primary distinction between Corollaries 2.18 and 2.21 is the way in which they capture the idea that neither measure is concentrated along lines, but it is the non-discrete phrasing in terms of measure along a common line that we used in \S \ref{ss:thin-tubes} above.

\begin{cor}[\cite{orponen2022kaufman} Corollary 2.21] \label{cor:cor2.21}
    If $\mu,\nu \in \mathcal{P}(\R^2)$ are $s$-Frostman measures for some $s \in (0,1]$ and if $\mu(\ell) \nu(\ell) < 1$ for every line $\ell \subset \R^2$, then $(\mu,\nu)$ have $\sigma$-thin tubes for all $\sigma < s$.
\end{cor}

\textit{Proof.} Our strategy is to work by cases according to the relationship between $\mu(\ell)$ and $\nu(\ell)$.

\textsc{Case 1.} Suppose that $\nu(\ell) > 0$ for some line $\ell \subset \R^2$ and that $\mu(\ell) < 1$. Any compact $K \subset \R^2 \setminus \ell$ with $\mu(K) > 0$ yields a pair $(\mu|_K,\nu|_\ell)$ with $s$-thin tubes (hence, $\sigma$-thin tubes for all $\sigma < s$): the measure $\pi_x \nu|_\ell$ for points $x \in K$ will always be $s$-Frostman because $\pi_x|_\ell \!: \ell \to \bbs^1$ is locally bi-Lipschitz, so Lemma \ref{lem:frostman-iff-tubes} applies.

\textsc{Case 2.} Suppose that $\mu(\ell) > 0$ for some line $\ell \subset \R^2$. Then $\nu(\ell) < 1$ by hypothesis, so there is a compact set $K \subset \R^2 \setminus \ell$ such that $\nu(K) > 0$.

The idea now is to appeal to OSW Lemma 2.22, which in turn cites Kaufman's proof of Marstrand's projection theorem. In that proof, one takes a $\sigma$-Frostman measure $\tilde{\nu}$ on the given set $E \subset \R^2$ ($0 < \sigma < s = \dim E$) and an $s$-Frostman measure $\tilde{\mu}$ on the exceptional set in $\bbs^1$ and proceeds to show that
\begin{equation} \label{eq:marstrand-energy}
    \int_{\bbs^1} I_\sigma(\pi_e\tilde{\nu}) \, d\tilde{\mu}(e) \lesssim I_\sigma(\tilde{\nu}) < \infty,
\end{equation}
whence $\dim \pi_e(E) \geq \sigma$ for $\tilde{\mu}$-a.e.~$e \in \bbs^1$. One derives \eqref{eq:marstrand-energy} from the geometric inequality
\begin{equation*}
    \tilde{\mu}\big( \{ e \in \bbs^1 \!: |\pi_e(y_1) - \pi_e(y_2)| \leq \delta \} \big) \lesssim \left( \frac{\delta}{|y_1 - y_2|} \right)^s,
\end{equation*}
and the same reasoning used to obtain this inequality gives the analogous inequality
\begin{equation*}
    \mu\big( \{ x \in \ell \!: |\pi_x(y_1) - \pi_x(y_2)| \leq \delta \} \big) \lesssim \left( \frac{\delta}{|y_1 - y_2|} \right)^s,
\end{equation*}
from which
\begin{equation*}
    \int_\ell I_\sigma(\pi_x\nu|_K) \, d\mu(x) \lesssim I_\sigma(\nu|_K) < \infty
\end{equation*}
follows by analogy with \eqref{eq:marstrand-energy}. See Mattila \cite{mattila2015fourier} Chapter 5 for a detailed account. This implies that $\pi_x\nu|_K$ is $\sigma$-Frostman for $\mu|_\ell$-a.e.~$x \in \ell$, so it follows from Lemma \ref{lem:frostman-iff-tubes} that $(\mu|_\ell,\nu|_K)$ has $\sigma$-thin tubes.


\textsc{Case 3.} Suppose that $\nu(\ell) = 0 = \mu(\ell)$ for every line $\ell \subset \R^2$. We reduce to the case of OSW Corollary 2.18, wherein the tube condition of Lemma \ref{lem:lem2.8} enters the picture. (Notice that this ``key lemma" was not used in either of the first two cases: the desired conclusion was reached by more general considerations.) To that end, assume without loss of generality (restricting and renormalizing the measures as necessary) that $\dist(\spt \mu, \spt \nu) > 0$. A short compactness argument shows that, since $\mu$ and $\nu$ vanish on every line, we have for every $\tau > 0$ some $\delta > 0$ such that
\begin{equation*}
    \max \, \{ \mu(T), \nu(T) \} \leq \tau \qquad \forall \+ \delta\text{-tubes } T \subset \R^2
\end{equation*}
(see \cite{orponen2018dimension} Lemma 2.1).

In particular, let $\tau = \tau(\bar{\epsilon},s) > 0$ and $\beta = \beta(s) > 0$ be as in Lemma \ref{lem:prop2.4}, where $\bar{\epsilon} = \bar{\epsilon}(\sigma,s) \in (0,\tfrac{1}{10})$ is such that $1 - 5^N \bar{\epsilon} \geq \tfrac{1}{2}$, $N$ satisfies $N \geq (\sigma-\beta)/\eta$, and $\eta = \eta(\sigma,s) > 0$ is as in Lemma \ref{lem:lem2.8}. Since $\eta$ is undefined if $\sigma < \beta$, we assume without loss of generality that $\sigma \geq \beta$. With $C > 0$ sufficiently large, Lemma \ref{lem:prop2.4} then implies that $(\mu,\nu)$ and (by symmetry) $(\nu,\mu)$ have $(\beta,K,1-\bar{\epsilon})$-thin tubes for some $\beta,K > 0$. Supposing we have shown that these pairs have $(\beta+k\eta, K', 1-5^k\bar{\epsilon})$-thin tubes for some integer $0 \leq k < N$ and some $K' > 0$, then, by our choice of $\bar{\epsilon}$ and $N$, the hypotheses of Lemma \ref{lem:prop2.4} are satisfied, so there are in fact $(\beta+(k+1)\eta, K'', 1-5^{k+1}\bar{\epsilon})$-thin tubes for some $K'' > 0$. By induction, $(\mu,\nu)$ and $(\nu,\mu)$ have $(\beta+N\eta, K, 1-5^N\bar{\epsilon})$-thin tubes for some $K > 0$. Since we chose $\beta+N\eta > \sigma$ and $1-5^N\bar{\epsilon} \geq \tfrac{1}{2}$, it follows that $(\mu,\nu)$ and $(\nu,\mu)$ have $(\sigma,K,\tfrac{1}{2})$-tubes, as desired. \hfill $\square$

The discussion at the end of \S \ref{ss:thin-tubes} presents the implication of the intended Theorem \ref{thm:radial-kaufman} from Corollary \ref{cor:cor2.21}. The only remaining hole in our argument is the proof of the ``key lemma."

\subsection[The epsilon-improvement of Wolff's Furstenberg set bound]{The \for{toc}{\texorpdfstring{$\epsilon$}{epsilon}}\except{toc}{\texorpdfstring{$\bm{\epsilon}$}{epsilon}}-improvement of Wolff's Furstenberg set bound} \label{ss:discrete-furstenberg}

The workhorse driving the proof of Lemma \ref{lem:lem2.8} is the discretized $\epsilon$-improvement of Wolff's celebrated Furstenberg set bound in \cite{wolff1999recent}. This discretized result in fact proves both the main results in \cite{orponen2021hausdorff} and merits appreciation in its own right. Let $|\+\cdot\+|_\delta$ denote the \textit{$\delta$-covering number}, i.e., the (minimum) number of balls of diameter $\delta$ required to cover a given set.

\begin{defn}\label{def:delta-s-set}
    Let $(X,d)$ be a metric space, let $\delta,C > 0$, and let $s \geq 0$. We call a set $P \subset X$ a \define{$\bm{(\delta,s,C)}$-set} if
    \begin{equation*}
        |P \cap B(x,r)|_\delta \leq C \+ r^s \+ |P|_\delta \qquad \forall \+ x \in X \text{ and } \forall r \+ \geq \delta.
    \end{equation*}
\end{defn}

When viewed at scales greater than $\delta$, such a set $P$ resembles the $\delta$-neighborhood of an (at least) $s$-dimensional set. In what follows, we will be especially interested in ``$(\delta,s,C)$-sets of tubes." The affine Grassmannian $\mathcal{A}(2,1)$ of lines in the plane enjoys the following metric space structure: if $\ell_i \in \mathcal{A}(2,1)$ has direction vector $e_i$ and closest point to the origin $a_i$ for $i = 1,2$, then the distance between $\ell_1$ and $\ell_2$ is defined by
\begin{equation*}
    d(\ell_1, \ell_2) := \| \pi_{\Span e_1} - \pi_{\Span e_2} \| + |a_1 - a_2| = \angle (e_1 ,e_2) + \vert a_1 - a_2 \vert
\end{equation*}
where $\| \cdot \|$ is the operator norm and $|\+\cdot\+|$ is the Euclidean distance. For more on the metric space structure of the affine Grassmannian, see Chapter 3 of \cite{mattila1995geometry}.

For a fixed $\delta$, we may then call a family of $\delta$-tubes a \textit{$(\delta,s,C)$-set of tubes} if the axes of the tubes form a $(\delta,s,C)$-set of lines, considered as a set in the metric space $\big( \mathcal{A}(2,1), d \big)$. With this language, the $\epsilon$-improvement of the Furstenberg set estimate can be phrased as follows.

\begin{thm}[\cite{orponen2022kaufman} Theorem 2.7, \cite{orponen2021hausdorff} Theorem 1.3] \label{thm:furstenberg}
    Given $\sigma \in (0,1)$ and $s \in (\sigma,2)$, there exist $\delta_0(\sigma,s), \epsilon(\sigma,s) > 0$ such that the following holds for all $0 < \delta \leq \delta_0$: if $X \subset B^2$ is a $(\delta,s,\delta^{-\epsilon})$-set of points and if for each $x \in X$ there is a $(\delta,\sigma,\delta^{-\epsilon})$-set $\mathcal{T}_x'$ of $\delta$-tubes containing $x$, then
    \begin{equation*}
        \left| \bigcup_{x \in X} \mathcal{T}_x' \right|_\delta \geq \delta^{-2\sigma-\epsilon}.
    \end{equation*}
    Moreover, $\epsilon$ can be taken uniform on compact subsets of $\{ (\sigma,s) \in \R^2 \!: \sigma \in (0,1), s \in (\sigma,2) \}$.
\end{thm}

The set $\bigcup_{x \in X} \mathcal{T}_x'$ should be understood as a discrete $(\sigma,s)$-Furstenberg set, and the conclusion of Theorem \ref{thm:furstenberg} is that these sets have ``dimension" at least $2\sigma + \epsilon$. See also \S \ref{ss:epsilon-improvement} below. The final statement about the uniformity of $\epsilon$ is crucial, as it in turn leads to the uniformity of $\eta$ in Lemma \ref{lem:lem2.8}\textemdash a detail that stars prominently in the proof of Corollary \ref{cor:cor2.21} (OSW Corollary 2.18).

\subsection{Proof of the key lemma} \label{ss:thm1.1-key-lemma}

The idea behind Lemma \ref{lem:lem2.8} is as follows. Suppose that the conclusion fails, say, that $(\mu,\nu)$ does not have $(\sigma+\eta)$-thin tubes with the prescribed parameters. A lengthy but rudimentary construction then produces a $(\sigma,s)$-Furstenberg set of tubes to which Theorem \ref{thm:furstenberg} applies. Roughly, this is done as follows. Begin with the ``bad tubes" afforded by our contradiction hypothesis. The number, sizes, and arrangement of the tubes are all unknown, but, by loosening our definition of ``bad" (essentially, relaxing the parameters $K$ and $c$ in Definition \ref{defn:thin-tubes-2}), we can construct a large collection of bad tubes of uniform width. A density argument shows that most of these tubes are ``non-concentrated," but this non-concentrated property implies that the tubes form a Furstenberg set. From there, a careful counting (``discrete energy") argument implies that the number of tubes is inconsistent with the $\nu$-measure required by our contradiction hypothesis.

It is finally time to work through the proof of Lemma \ref{lem:lem2.8} rigorously. Our purpose here is twofold: to break the argument down into steps, and to fill in details omitted in \cite{orponen2022kaufman}. Neither purpose supplants that of the original proof, through which the reader is encouraged to work carefully sooner than later. To make the dependencies clear among the various constants that arise, we frequently refer to Table \ref{tab:parameters} below.

\define{Proof of Lemma \ref{lem:lem2.8}.} \textsc{Step 1.} We set up the proof by contradiction. Since $\mu$ and $\nu$ play symmetric roles in the statement and conclusion, suppose for a contradiction that $(\mu,\nu)$ do not have $(\sigma+\eta, K', 1-5\epsilon)$-thin tubes, where $\eta$ and $K'$ are explicit functions of $K$, $C$, $s$, and $\sigma$. (These ``explicit functions" depend on $\epsilon$\textemdash which we suppress due to its eventual dependence on $s$ and $\sigma$ in Corollary \ref{cor:cor2.21}\textemdash and on several other parameters introduced on p.~11 of OSW. The important point is that one could in principle prescribe its value at the very beginning of the proof; cf.~Table \ref{tab:parameters}.) Letting $X := \spt \mu$ and $Y := \spt \nu$, we may apply the definition of $(\sigma, K, 1-\epsilon)$-thin tubes to obtain sets $G' \subset X \times Y$ and $G'' \subset Y \times X$ as in the definition of thin tubes for $(\mu,\nu)$ and $(\nu,\mu)$, respectively. Swapping $x$- and $y$-variables in $G''$ and intersecting with $G'$ yields a set $G \subset X \times Y$ with $(\mu \times \nu)(G) > 1 - 2\epsilon$ such that
\begin{equation} \label{eq:sigma-thin-tubes}
\begin{aligned}
    & \nu(T \cap G|_x) \leq K \+ r^\sigma \quad \forall \+ r > 0 \ \text{ and } \ \forall \+ r\text{-tubes } T \text{ containing } x, \\
    & \mu(T \cap G|^y) \leq K \+ r^\sigma \quad \forall \+ r > 0 \ \text{ and } \ \forall \+ r\text{-tubes } T \text{ containing } y.
\end{aligned}
\end{equation}

\textsc{Step 2a.} We begin the construction of a ``tame" family of ``bad" tubes. For each $x \in X$ and dyadic $r \leq 1$, begin with the family $\mathcal{T}_{x,r}''$ of $r$-tubes $T$ supporting our contradiction hypothesis, i.e., such that
\begin{equation} \label{eq:reverse-frostman}
    \nu(T \cap G|_x) \geq K' \+ r^{\sigma+\eta}.
\end{equation}
\textit{Some} such tube exists by hypothesis, but not necessarily with $r$ dyadic\textemdash a possibility we contend with later. To pass from $\mathcal{T}_{x,r}''$ to a ``well-behaved" surrogate, start with an $r$-separated family $\mathcal{T}^r$ of $2r$-tubes with $|\mathcal{T}^r| \sim r^{-2}$ such that the following holds: for each $r$-tube $T \subset \R^2$, the intersection $T \cap B^2$ lies in at least $1$ and at most (say) $50$ tubes in $\mathcal{T}^r$. Such a family exists by picking an $r$-net of directions, and a tiling of space by parallel $2r$-tubes, one tiling for each direction in the net.

Next we pass from $\mathcal{T}^r$ to a minimal subset $\mathcal{T}_{x,r}' \subset \mathcal{T}^r$ with the property that, for each $T'' \in \mathcal{T}_{x,r}''$, there exists $T' \in \mathcal{T}_{x,r}'$ such that $T'' \cap B^2 \subset T'$. The only reason for bringing $\mathcal{T}^r$ into the picture was to demonstrate how to choose $\mathcal{T}_{x,r}'$ to be $r$-separated. Notice that, whereas $|\mathcal{T}^r| \sim r^{-2}$, we have
\begin{equation} \label{eq:number-of-tubes}
    |\mathcal{T}_{x,r}'| \lesssim r^{-\sigma-\eta}.
\end{equation}
To see this, notice that, since each tube in $\mathcal{T}_{x,r}'$ contains a tube $T$ satisfying \eqref{eq:reverse-frostman}, that inequality holds for all $T \in \mathcal{T}_{x,r}'$ as well. (Although $\mathcal{T}_{x,r}'$ is a family of $2r$-tubes, we do not replace the $r$ in \eqref{eq:reverse-frostman} with $2r$.) Since $\mathcal{T}_{x,r}'$ at worst forms (say) a $10C$-fold cover of $\spt \nu$ (where $C$ is taken sufficiently large that $\tfrac{1}{C} \leq \dist(\spt \mu, \spt \nu)$), the sum of the $\nu$-masses of the tubes is $\leq 10C \sim 1$, from which \eqref{eq:number-of-tubes} follows. (This $10C$\textemdash an unimportant constant\textemdash follows from OSW Equation (2.12) with implicit constant $10$ by taking the infimum over all $y \in \spt \nu$. See \cite{mattila1995geometry} Chapter 3 and \cite{mattila2015fourier} Chapter 5.)

\textsc{Step 2b.} To accompany our family $\mathcal{T}_{x,r}'$ of bad tubes, we construct a set $H$ to serve as a substitute for $G$. As the argument in OSW is quite thorough, we do not reproduce the logic but simply recall the conclusions: for dyadic $r \in (0,1]$ and
\begin{equation*}
    H_r := \left\{ (x,y) \in G \!: y \in \bigcup \mathcal{T}_{x,r}' \right\}, \qquad H := \bigcup_{r \in 2^{-\N}} H_r,
\end{equation*}
we have $(\mu \times \nu)(H) \geq 1 - 3\epsilon$. Note that one can dispense with the rescaling of $K'$ by replacing the $r$-tube $T$ in their argument with a $2r$-tube in some $\mathcal{T}_{x,r}'$, and this also makes the chain of implications somewhat clearer.

\textsc{Step 3.} We begin constructing a discretized Furstenberg set with base points taken from a set $\upbold{X}$ (defined below) and tubes taken from the families $\mathcal{T}_{x,r}'$, $x \in \upbold{X}$. Let $r_0$, $r_1$, and $r_2$ be as in Table \ref{tab:parameters}. Since $K' \+ r_0^{\sigma+\eta} \geq 1$ and $|\nu| = 1$, Equation \eqref{eq:reverse-frostman} cannot hold for any $r > r_0$. Hence, $\mathcal{T}_{x,r}' = \varnothing$ for all such $r$ and, since $H_r |_x \subset \bigcup \mathcal{T}_{x,r}'$, the set $H_r$ is likewise empty. Writing
\begin{equation*}
    H = \bigcup_{r_0 \geq r \in 2^{-\N}} H_r
\end{equation*}
and recalling that
\begin{equation*}
    \sum_{r_0 \geq r \in 2^{-\N}} r^\eta < \frac{\epsilon}{2}
\end{equation*}
by our choice of $r_2$ and $r_0 \leq r_2$, we conclude from the lower bound $(\mu \times \nu)(H) \geq \epsilon$ that there exists some dyadic ${\rm r} \leq r_0$ such that $(\mu \times \nu)(H_{{\rm r}}) \geq 2\+{\rm r}^{\sigma+\eta}$. (Breaking from the notation of OSW, we distinguish this scale ${\rm r}$ from arbitrary values of $r \in (0,1]$ by omitting the italics. However, to avoid confusion, we hereafter use $\rho$ as a dummy index where normally we might use $r$.)

With this value of ${\rm r}$ in hand, let
\begin{equation*}
    \upbold{X} := \{ x \in X \!: \nu(H_{\rm r} |_x) \geq {\rm r}^\eta \}
\end{equation*}
and, for $x \in \upbold{X}$,
\begin{equation*}
    \mathcal{T}_x := \big\{ T \in \mathcal{T}_{x,{\rm r}}' \!: \nu(T \cap H_{\rm r} |_x) \geq {\rm r}^{\sigma + 3\eta} \big\} \quad \text{and} \quad Y_x := (H_{\rm r} |_x) \cap \bigcup \mathcal{T}_x. 
\end{equation*}
To obtain $Y_x$ from $H_{\rm r} |_x$, we are subtracting off at most $|\mathcal{T}_{x,{\rm r}}'| \lesssim {\rm r}^{-\sigma-\eta}$ tubes, each of mass $< {\rm r}^{\sigma + 3\eta}$, per the definition of $\mathcal{T}_x$. Since $\nu(H_{\rm r} |_x) \geq {\rm r}^\eta$ by our choice of $x \in \upbold{X}$, it follows that
\begin{equation} \label{eq:yx-lower-bd}
    \nu(Y_x) \gtrsim {\rm r}^\eta - {\rm r}^{-\sigma-\eta} \cdot {\rm r}^{\sigma + 3\eta} = {\rm r}^\eta - {\rm r}^{2\eta} \gg {\rm r}^{2\eta}.
\end{equation}
In addition, the relations $T \cap Y_x = T \cap H_{\rm r} |_x \subset T \cap G|_x$ and the thin tube condition \eqref{eq:sigma-thin-tubes} on $(\mu,\nu)$ give
\begin{equation} \label{eq:osw-eq2.15}
    {\rm r}^{\sigma + 3\eta} \leq \nu(T \cap H_{\rm r} |_x) \leq \nu(T \cap G|_x) \leq K\+{\rm r}^\sigma \leq r_0^{-\eta} {\rm r}^\sigma \leq {\rm r}^{\sigma - \eta},
\end{equation}
$T \in \mathcal{T}_x$. (Recall the definitions of $r_0$ and ${\rm r}$.) Using the lower bound \eqref{eq:yx-lower-bd} and summing the upper bound \eqref{eq:osw-eq2.15} over all $T \in \mathcal{T}_x$ yields
\begin{equation*}
    {\rm r}^{-\sigma+3\eta} \lesssim {\rm r}^{-\sigma + \eta} \+ \nu(Y_x) \lesssim |\mathcal{T}_x| \leq |\mathcal{T}_{x,{\rm r}}'| \lesssim {\rm r}^{-\sigma-\eta}.
\end{equation*}
As noted in OSW, the inequality $\nu\big( T^{(\rho)} \cap Y_x \big) \lesssim \rho^\sigma {\rm r}^{-\eta}$ holds for all $\rho \in [r,1]$ and $T \in \mathcal{T}_x$, where $T^{(\rho)}$ is the $\rho$-neighborhood of $T$ (i.e., an $(r+\rho)$-tube). Denoting by $\mathcal{L}_x \subset \mathcal{A}(2,1)$ the set of axial lines of the tubes in $\mathcal{T}_x$, we conclude that, for all $\ell \in \mathcal{A}(2,1)$, the inequalities
\begin{align*}
    |\mathcal{L}_x \cap B(\ell,\rho)|_{\rm r} &\leq \big| \big\{ \ell' \in \mathcal{L}_x \!: \ell' \cap B^2 \subset \ell^{(\rho)} \big\} \big|_{\rm r} \sim \big| \big\{ T' \in \mathcal{T}_x \!: T' \cap B^2 \subset T^{(\rho)} \big\} \big| \\
    &\leq \rho^\sigma {\rm r}^{-\sigma-2\eta} \lesssim \rho^\sigma {\rm r}^{-5\eta} |\mathcal{T}_x| \sim \rho^\sigma {\rm r}^{-5\eta} |\mathcal{L}_x|_{\rm r}
\end{align*}
hold up to a rescaling of the metric on $\mathcal{A}(2,1)$. (Note that, since $\mathcal{L}_x$ is ${\rm r}$-separated, its ${\rm r}$-covering number is proportional to its cardinality.) This is precisely what it means for $\mathcal{T}_x$ to be an $({\rm r}, \sigma, {\rm r}^{-5\eta})$-set of tubes.

\textsc{Step 4.} We divide the analysis of $\bigcup_{x \in \upbold{X}} \mathcal{T}_x$ into two cases. With $\kappa$ as in Table \ref{tab:parameters}, say that a tube $T \in \mathcal{T}_x$ is \textit{concentrated} if there is an ${\rm r}^\kappa$-ball $B_T \subset \R^2$ such that
\begin{equation} \label{eq:concentrated}
    \nu(T \cap B_T \cap Y_x) \geq \frac{1}{3} \+ \nu(T \cap Y_x).
\end{equation}
This means that at least one third of the $\nu$-mass of $T$ is contained in a fairly small ball (with radius not much greater than that of $T$)\textemdash a condition that holds when $T$ intersects $Y_x$ ``transversely."

This leads to our first case. Suppose there is a set $\upbold{X}' \subset \upbold{X}$ with $\nu(\upbold{X}') \geq \tfrac{1}{2} \nu(\upbold{X}) \geq \tfrac{1}{2} \+ {\rm r}^\eta$ such that, for all $x \in \upbold{X}'$, the set $\mathcal{T}_x' \subset \mathcal{T}_x$ of tubes that are \textit{not} concentrated (i.e., such that \eqref{eq:concentrated} \textit{fails} for every ${\rm r}^\kappa$-ball $B_T$) comprises at least half of $\mathcal{T}_x$. Since $\mu$ is $s$-Frostman, the quantitative form of the mass distribution principle yields the lower bound $\mathcal{H}_\infty^s(\upbold{X}') \gtrsim \tfrac{1}{C} \+ {\rm r}^\eta$ on the Hausdorff content of $\upbold{X}'$.

We can now apply the discrete Frostman's lemma of F\"{a}ssler--Orponen (Lemma \ref{lem:discrete-frostman} below) with $A = \upbold{X}'$ and $\rho = {\rm r}$ to obtain the stipulated $({\rm r}, s, {\rm r}^{-2\eta})$-set $P \subset \upbold{X}'$. Since each $\mathcal{T}_x$ is an $({\rm r}, \sigma, {\rm r}^{-5\eta})$-set of tubes and $|\mathcal{T}_x'| \geq \tfrac{1}{2} |\mathcal{T}_x|$, each $\mathcal{T}_x$ is an $({\rm r}, \sigma, 2{\rm r}^{-5\eta})$-set of tubes. Thus, $\mathcal{T}' := \bigcup_{x \in P} \mathcal{T}_x'$ is a discretized $(\sigma,s)$-Furstenberg set. The ${\rm r}$-separation property of $\mathcal{T}'$ then combines with Theorem \ref{thm:furstenberg} to produce the estimate
\begin{equation} \label{eq:furstenberg}
    |\mathcal{T}'| \gtrsim {\rm r}^{-2\sigma - \sqrt{\eta}},
\end{equation}
as $\sqrt{\eta} \gg 5\eta$.

Removing a well-chosen ${\rm r}^\kappa$-ball from each $T \in \mathcal{T}'$, we find an ${\rm r}^\kappa$-separated pair $Y_{T,1}, Y_{T,2} \subset T$ with $\nu(Y_{T,j}) \geq \tfrac{1}{3} {\rm r}^{\sigma+3\eta} \geq {\rm r}^{\sigma+4\eta}$. The remainder of the proof that this non-concentrated property leads to a contradiction is a straightforward computation to which the reader is referred in OSW. The upshot is that, on elementary geometric grounds, there cannot be many well-spaced tubes containing a given pair of not-too-close points, and this basic upper bound on the number of tubes containing points of $Y_{T,1}$ and $Y_{T,2}$ is enough to contradict \eqref{eq:furstenberg}. For further examples and diagrams of this classic ``two-ends argument," see \cite{orponen2015packing} Theorems 1.6 and 1.8; or, for the higher-dimensional analog, \cite{bushling2022packing}.

\textsc{Step 5.} By Step 4, there is a set $\upbold{X}' \subset \upbold{X}$ with $\mu(\upbold{X'}) \geq \tfrac{1}{2} \mu(\upbold{X})$ such that, for all $x \in \upbold{X}'$, the set $\mathcal{T}_x' \subset \mathcal{T}_x$ of tubes that \textit{are} concentrated comprises at least half of $\mathcal{T}_x$. With
\begin{equation*}
    H' := \big\{ (x,y) \in \R^2 \!: x \in \upbold{X}', \ \exists \+ T \in \mathcal{T}_x' \text{ s.t.}\ y \in T \cap B_T \cap Y_x \big\},
\end{equation*}
we have the lower bound $(\mu \times \nu)(H') \gtrsim {\rm r}^{7\eta}$. To see this, observe that one can approximate $(\mu \times \nu)(H')$ by considering each $x \in \upbold{X}'$, summing $\nu(T \cap B_T \cap Y_x)$ for each $T \in \mathcal{T}_x'$, and integrating with respect to $d\mu(x)$. The one issue is that each $y \in Y_x$ can lie in many distinct tubes, but this overcounting does not pose an issue because this number is bounded by the constant $10C$ (cf.~Step 2a). The ``undergraph" estimate
\begin{equation*}
    (\mu \times \nu)(H') \geq \frac{1}{10C} \mu(\upbold{X}') \cdot \inf_{x \in \upbold{X}'} |\mathcal{T}_x'| \cdot \inf_{x \in \upbold{X}', T \in \mathcal{T}_x'} \nu(T \cap B_T \cap Y_x)
\end{equation*}
follows at once, from which the lower bound $\sim{\rm r}^{7\eta}$ follows by the estimates in Step 3.

The trick now is to locate an annulus centered about a point of $Y$ such that the density of the annulus conflicts with the Frostman condition on $\nu$. We first identify the correct radius of the annulus. Notice that, if $(x,y) \in H'$, then $B_T \subset B(y,2{\rm r}^k)$. This pairs with the concentration hypothesis on $\mathcal{T}_x'$ to give \newpage
\begin{align*}
    \nu\big( T(x,y) \cap B(y,2{\rm r}^\kappa) \big) &\geq \nu\big( T(x,y) \cap B(y,2{\rm r}^\kappa) \cap Y_x \big) \\
    &\geq \nu\big( T(x,y) \cap B_{T(x,y)} \cap Y_x \big) \\
    &\geq \tfrac{1}{3} \nu(T(x,y) \cap Y_x) \geq \tfrac{1}{3} r^{\sigma+3\eta},
\end{align*}
where the final bound comes \eqref{eq:osw-eq2.15}. On the other hand, since we chose $3\eta < s - \sigma$ and ${\rm r}^\eta \leq r_1 \leq \tfrac{1}{6C}$ (cf.~Table \ref{tab:parameters}),
\begin{equation*}
    \nu(B(y,{\rm r})) \leq C \+ {\rm r}^s < C \+ {\rm r}^{\sigma + 3\eta} \leq C \+ r_0^\eta \+ {\rm r}^{\sigma + 3\eta} \leq \tfrac{1}{6} {\rm r}^{\sigma + 3\eta} \leq \tfrac{1}{2} \nu\big( T(x,y) \cap B(y,2{\rm r}^\kappa) \big).
\end{equation*}
Thus, we have a discrepancy in the $\nu$-measure in balls at different orders of magnitude: large balls contain disproportionately more mass than small ones, so there exists a dyadic scale $\xi(x,y)$ between the small scale ${\rm r}$ and the large scale ${\rm r}^\kappa$ at which the annulus $A(y,\xi(x,y),$ $2\xi(x,y))$ contains mass not too much less than that of $T(x,y) \cap B(y,2{\rm r}^\kappa)$\textemdash say,
\begin{equation*}
    \nu\big( A(y,\xi(x,y),2\xi(x,y)) \big) \geq {\rm r}^{\sigma+4\eta}.
\end{equation*}
We showed that the number of pairs $(x,y) \in H'$ is $\gtrsim {\rm r}^{7\eta}$, which is much greater than ${\rm r}^\kappa - {\rm r} \sim {\rm r}^\kappa$, so the elementary pigeonhole principle yields a $\xi \in [{\rm r}, {\rm r}^\kappa]$ independent of $x$ and $y$ such that
\begin{equation*}
    (\mu \times \nu)(H'') \geq {\rm r}^{8\eta}, \quad \text{where} \quad H'' := \{ (x,y) \in H' \!: \xi(x,y) = \xi \}.
\end{equation*}
This is more than we need, as we require a single $y \in Y$ such that $\mu(H''|^y) \geq {\rm r}^{8\eta}$.

\vs{0.3}

\begin{table}[h!]
\centering
\begin{tabular}{cl}
    \multicolumn{2}{c}{\textbf{Table of Constants}} \\[0.2cm]
    \hline \\[-0.25cm]
    \textbf{Symbol} & \textbf{Specification} \\[0.15cm]
    $\epsilon$ &  \multicolumn{1}{p{9.5cm}}{\raggedright for the purposes of this argument, an arbitrarily small number $< \tfrac{1}{10}$} \\[0.55cm]
    $\eta = \eta(\sigma,s)$ & \multicolumn{1}{p{9.5cm}}{\raggedright a sufficiently small constant in terms of $\sigma$ and $s$, say, \hs{1.9} $\min \+ \left\{ \epsilon(\sigma,s), \+ \tfrac{s-\sigma}{4}, \+ \tfrac{1}{2} \! \left( \frac{s-\sigma}{14 - 8(s-\sigma)} \right)^2 \right\}, \hs{1.9}$ where $\epsilon(\sigma,s)$ is as in Theorem \ref{thm:furstenberg} (and \textit{not} the $\epsilon$ above)} \\[1.35cm]
    $K' = K'(K,C,\sigma,s)$ & \multicolumn{1}{p{9.5cm}}{\raggedright a sufficiently large constant in terms of $K$, $C$, $\sigma$, and $s$, say, \hs{2.75} $\max \+ \left\{ K^{1/\eta}, \+ \tfrac{1}{r_2}, \+ \tfrac{1}{r_0^{\sigma+\eta}} \right\}$} \\[0.8cm]
    $\kappa = \kappa(\sigma,s)$ & $\tfrac{14\eta}{s-\sigma}$ \\[0.15cm]
    ${\rm r} = {\rm r}(K,C,\sigma,s)$ & a dyadic number $\leq r_0$ such that $(\mu \times \nu)(H_{{\rm r}}) \geq 2\+{\rm r}^\eta$ \\[0.15cm]
    $r_0 = r_0(K,C,\sigma,s)$ & $\min \big\{ K^{-1/\eta}, r_1, r_2 \big\}$ \\[0.15cm]
    $r_1 = r_1(C,\sigma,s)$ & \multicolumn{1}{p{9.5cm}}{\raggedright a dyadic scale less than $\tfrac{1}{(6C)^{1/\eta}}$ at which the discrete Frostman lemma (\cite{fassler2014restricted} Lemma A.1) can be applied} \\[0.55cm]
    $r_2 = r_2(\sigma,s)$ & \multicolumn{1}{p{9.5cm}}{\raggedright a sufficiently small number that $\sum_{r_2 \geq r \in 2^{-\N}} r^\eta < \tfrac{\epsilon}{2}$, say, $2^{(\log \eta\epsilon - 2)/\eta}$; its only use is in defining $r_0$}
\end{tabular}
\vs{0.6}
\caption{A table of all the constants introduced in the proof of the ``key" Lemma \ref{lem:lem2.8}, the main hurdle in obtaining Theorem \ref{thm:radial-kaufman}.}
\label{tab:parameters}
\end{table}

\vs{-0.15} Although the remainder of the proof is daunting at a glance, enough detail is presented in clear enough language that we refer the reader to OSW to complete Step 5 and, hence, the proof of the key lemma. We only lay out the main point, which is this: the contradiction in the geometry of $\upbold{X}'$ arises from a conflict between the thin tubes condition on $(\nu,\mu)$\textemdash the only time this hypothesis is used\textemdash and the $s$-Frostman condition on $\mu$. Naturally, balls and tubes have very different geometries, so this requires a little work, but the setup is already in place.

We selected a slice $H''|^y$ with large $\mu$-mass, and $H''$ was defined so that $H''|^y$ is covered by tubes through $y$ with mass concentrated in a very specific region\textemdash in an annulus centered at $y$ with inner radius $\xi$ and outer radius $2\xi$. Application of the thin tube condition to each of these tubes supplies an upper bound on the mass of the tubes and, hence, a lower bound on the number of tubes. On the other hand, since these tubes span a well-spaced set of directions about $y$, the $\nu$-mass of these tubes in $B(y,2\xi)$ is bounded in a nontrivial way by $\nu(B(y,2\xi))$ itself. It is this ball to which the Frostman bound applies, yielding an upper bound on the number of tubes inconsistent with the lower bound.

This proves that the set $\upbold{X}$ constructed in Step 3 cannot exist, whereupon we must reject the hypothesis that $(\mu,\nu)$ do not have $(\sigma+\eta,K',1-\epsilon)$-thin tubes. \hfill $\square$

\subsection[The role of epsilon]{The role of \for{toc}{\texorpdfstring{$\epsilon$}{epsilon}}\except{toc}{\texorpdfstring{$\bm{\epsilon}$}{epsilon}}} \label{ss:epsilon-improvement}

On the one hand, the role of $\epsilon$ in the proof of Lemma \ref{lem:lem2.8} is clear: the increment $\eta$ is some fraction of $\epsilon$, and this particular value is the ``right" one for the self-evident (and evidently unsatisfying) reason that the non-existence of $(\sigma+\eta)$-thin tubes entails the existence of a set of tubes with ``dimension" lower than what is permitted by the Furstenberg set bound. In this capacity, there is nothing so magical about the $\epsilon$. The important fact is that, if we \textit{do not} have thin tubes, then there is a family of ``bad" tubes of small ``dimension." This family of tubes may be concentrated or non-concentrated, and in the latter case, any sufficiently strong incidence estimate for tubes will lead to a contradiction. The cutoff for ``sufficiently strong" just so happens to be Wolff's bound from \cite{wolff1999recent}, with any strictly stronger estimate sufficing for our purpose.

On the other hand, that a mere $\epsilon$ makes the difference between truth or falsity of Theorem \ref{thm:radial-kaufman}\textemdash a \textit{qualitative} distinction\textemdash is perhaps bewildering. Why could we not weave another small parameter into the argument (say, in the Frostman condition on $\nu$) and exploit that margin for error so that we might work with Wolff's Furstenberg set bound instead? In fact, one \textit{could}, but the uniform control of $\eta$ in $\sigma$ and $s$ would be lost. It was this uniformity that guaranteed the convergence of the sequence
\begin{gather*}
    \beta + \eta(\beta,s), \ \beta + \eta(\beta,s) + \eta(\beta + \eta(\beta,s), s), \\
    \beta + \eta(\beta,s) + \eta(\beta + \eta(\beta,s), s) + \eta\big( \beta + \eta(\beta,s) + \eta(\beta + \eta(\beta,s), s), s \big), \ ...
\end{gather*}
to $s$ and not to some smaller value.

\section{The Falconer-type estimate, Theorem 1.2} \label{s:Section 3}

The following result is the ``Falconer-type" radial projection theorem of OSW.
\begin{thm}[\cite{orponen2022kaufman} Theorem 1.2]
    Let $X,Y\subset\R^2$ be Borel sets with $X\ne\varnothing$ and $\dim Y > 1$. Then
    \[
    \sup_{x\in X}\dim \pi_x(Y\setminus\{x\}) \ge \min \{\dim X + \dim Y - 1, 1\}.
    \]
\end{thm}
The theorem provides a strong lower bound for the dimension of radial projections. One can view part of the challenge of proving OSW Theorem 1.2 is finding a condition which guarantees under our assumptions that for each $\sigma < \min \{\dim X + \dim Y - 1, 1\}$, there exists $x\in X$ such that $\pi_x\nu$ satisfies a $\sigma$-dimensional Frostman condition. With hindsight, the definition of $\sigma$-thin tubes captures this precisely (see Lemma \ref{lem:KEY thin tubes}).
\begin{rem}\label{rem:start thin tubes}
    Another consequence of the definition of thin tubes is that in the setup of Theorem 1.2, we immediately have the weaker conclusion that the supremum appearing on the left-hand side is at least $\dim Y-1>0$.
\end{rem}
\begin{proof}
    For arbitrary $1 < t < \dim Y$ and $s<\dim X$, choose a $t$-Frostman measure $\nu$ supported in $Y$, and an $s$-Frostman measure $\mu$ supported in $X$. We show $(\mu,\nu)$ has $(t-1)$-thin tubes. Take $G = X\times Y$. For each $x\in X$, and each $r$-tube $T$ containing $x$, $T$ may be covered by $\sim r^{-1}$-many $r$-balls, so by the $t$-Frostman condition on $\nu$,
    \[
    \nu(T\cap G|_x) = \nu(T) \lesssim r^{-1}\cdot r^{t} = r^{t-1}.
    \]
    By Lemma \ref{lem:KEY thin tubes}, $\sup_{x\in X}\dim\pi_x(Y\setminus\{x\})\ge t-1$. As $t < \dim Y$ was arbitrary, the conclusion follows.
\end{proof}
 The innovation of the authors' argument is realizing the value of $\sigma$ in ``$(\sigma,K,c)$-thin tubes'' can be improved by a small but uniform amount over and over, repeatedly, at the cost of a larger value of $K$ and a smaller (but still positive) value of $c$.

 The thin-tubes version of OSW Theorem 1.2 is the following:
\begin{thm}\label{thm:theorem 1.2 thin tubes}
    Let $s\in[0,2],t\in(1,2]$, $0\le \sigma < \min \{s+t-1,1\}$, $C>0$ and $\epsilon\in(0,1]$. Then there exists $K = K(s,t,\sigma,C,\epsilon)$ such that the following holds. Assume $\mu,\nu\in\mathcal P(B^2)$ satisfy $\mu(B(x,r))\le Cr^s$ and $\nu(B(x,r))\le Cr^t$, or alternatively $I_s(\mu)\le C$ and $I_t(\nu)\le C$. Then $(\mu,\nu)$ has $(\sigma, K, 1-\epsilon)$-thin tubes.

    In particular, whenever $s\in[0,2]$, $t\in(1,2]$, $I_s(\mu)<\infty$ and $I_t(\nu) < \infty$, then $(\mu,\nu)$ has $\sigma$-thin tubes for every $0\le \sigma < \min \{s + t - 1, 1\}$.
\end{thm}
Since $\sigma < \min \{\dim X + \dim Y - 1, 1\}$ can be made arbitrarily close to the endpoint, by Lemma \ref{lem:KEY thin tubes}, Theorem \ref{thm:theorem 1.2 thin tubes} immediately gives us the proof of OSW Theorem 1.2. On the other hand, Remark \ref{rem:start thin tubes} says that we start out with $t-1>0$-thin tubes, for any $1<t<\dim Y$.

The Key Lemma of the authors' paper is a precise form of what we've said that the value ``$\sigma$'' of thin tubes can be improved over and over, repeatedly. Let's recall the statement of the Key Lemma for the proof of the Falconer-type radial projection theorem:
\begin{lem}[\cite{orponen2022kaufman} Lemma 3.21, The ``key lemma"] \label{lem:falconer epsilon improvement}
Let $s\in[0,2]$, $t\in(1,2]$, and $0\le \sigma < \min \{s + t - 1, 1\}$. Let $\epsilon\in(0,\frac1{10})$ and $C, K> 0$. Let $\mu,\nu\in\mathcal P(B^2)$ such that $\mu(B(x,r)) \le Cr^s$ and $\nu(B(y,r)) \le Cr^t$ for all $x,y\in\R^2$ and $r > 0$. If $(\mu,\nu)$ has $(\sigma, K, 1-\epsilon)$-thin tubes, there exist $\eta=\eta(s,t,\sigma) > 0$ and $K' = K'(s,t,\sigma,\epsilon,C,K)>0$ such that $(\mu,\nu)$ has $(\sigma+\eta,K',1-4\epsilon)$-thin tubes. Moreover $\eta(s,t,\sigma)$ is bounded away from zero on any compact subset of 
\[
\Omega = \{(s,t,\sigma)\in[0,2]\times(1,2]\times[0,1):\sigma < \min \{s+t-1,1\}\}.
\]
\end{lem}
Within the statement of the Key Lemma, the most important features are:
\begin{itemize}
    \item For fixed $s\in[0,2]$ and $t\in(1,2]$, the value $\eta(s,t,\sigma)$ can be taken uniformly positive for all $\sigma\in[t-1,\min\{s+t-1,1\}]$.
    \item We improve $\sigma$ to $\sigma + \eta$.
    \item It only costs us $3\epsilon$ in the third parameter of thin tubes.
\end{itemize}
With the Key Lemma in hand, the proof of the main technical theorem regarding thin tubes for the Falconer-type estimate is a very short and elegant proof by bootstrapping which we don't need to repeat here. Instead, we will focus on discovering what goes into the proof of Lemma \ref{lem:falconer epsilon improvement}.

\subsection{A zoomed-out look at the key lemma for the Falconer-type Theorem 1.2}
In this subsection, we will describe in less precise terms what goes into the Key Lemma for the Falconer-type Theorem 1.2. The proof of the Key Lemma is by contradiction, so we assume $(\mu,\nu)$ has $\sigma$-thin tubes, but not $(\sigma+\eta)$-thin tubes. 

The first big takeaway of this assumption is that this setup essentially immediately lets us find a dual $(s,\sigma)$-Furstenberg set of $r$-tubes $\mathcal T$ with base points in a large fraction of $X$, which we won't distinguish from $X$ at this point. Pretending that for some small scale $r>0$, $X$ is an $r$-separated $(r,s)$-set of points and $Y$ is an $r$-separated $(r,t)$-set of points, for each $x\in X$, we can find a $(r,\sigma)$-set $\mathcal T_x$ of $1\times r$ ``\textit{bad}'' tubes $\mathcal T_x$ containing $x$, such that
\begin{equation}\label{eq:disc not sigma+eta thin}
|P_Y\cap T| \ge r^{\sigma + \eta}|P_Y|,\qquad T\in\mathcal T_x.
\end{equation}
Inequality \eqref{eq:disc not sigma+eta thin} represents our assumption that $(\mu,\nu)$ does \emph{not} have $(\sigma+\eta)$-thin tubes, expressed in terms of the discretized sets $P_X$ and $P_Y$ and the bad tubes $\mathcal T_x$ which fail to satisfy the assumption of $(\sigma+\eta)$-thin tubes. 

This will lead to a contradiction by the following immediate consequence of the Fu--Ren incidence theorem, recorded in Theorem 3.1 of \cite{orponen2022kaufman} (with notation adapted for our use here):

\newpage

\begin{thm}[\cite{orponen2022kaufman} Theorem 3.1]\label{thm:fu-ren}
    For every $t\in (1,2]$, $\sigma\in[0,1)$, and $\zeta > 0$, there exist $\eta=\eta(t,\sigma,\zeta)>0$ and $r_0 = r_0(t,\sigma,\zeta)>0$ such that the following holds for all $r\in(0,r_0]$.

    Let $s\in[0,2]$. Let $P_Y\subset B^2$ be an $r$-separated $(r,t,r^{-\eta})$-set, and let $P_X\subset B^2$ be an $r$-separated $(r,s,r^{-\eta})$-set. Assume that for every $x\in P_X$, there exists a $(r,\sigma,r^{-\eta})$-set of tubes $\mathcal T_x$ with the properties $x\in T$ for all $T\in\mathcal T_x$, and
    \[
    |T\cap P_Y| \ge r^{\sigma+\eta}|P_Y|,\qquad T\in\mathcal T_x.
    \]
    Then $\sigma \ge s + t - 1 - \zeta$.
\end{thm}

The contradiction comes by essentially choosing $\eta$ to be the same $\eta$ from Theorem \ref{thm:fu-ren}. Since $\sigma < s + t - 1$, in particular, for some small $\zeta > 0$, we also have $\sigma < s + t - 1 - \zeta$. But now we see that by assuming we do not have $(\sigma+\eta)$-thin tubes, as long as we are able to work down from $\mu$ and $\nu$ to sets $P_X$ and $P_Y$ as above, and find the collections $\mathcal T_x$ of tubes for $x\in P_X$, Theorem 3.1 implies $\sigma \ge s + t - 1 - \zeta$, and this is a contradiction. The conclusion is that $(\mu,\nu)$ \emph{do} have $(\sigma+\eta)$-thin tubes, where $\eta$ is essentially the same number that appears in Theorem 3.1.

\subsection{A close look at the key lemma for the Falconer-type Theorem 1.2}

We will follow along with the authors as they prove Lemma \ref{lem:falconer epsilon improvement}. The proof starts by contradiction: we assume that we start with a pair $(\mu,\nu)$ that has $(\sigma,K,1-\epsilon)$-thin tubes, but $(\mu,\nu)$ does \emph{not} have $(\sigma + \eta, K',1-4\epsilon)$-thin tubes, for $\eta$ and $K'$ that need to be determined at some point during the proof. Here we can highlight another facet of the definition of thin tubes which is easy to overlook. If $(\mu,\nu)$ does not have $(\sigma+\eta,K',1-4\epsilon)$-thin tubes, then in particular, for \emph{every} $G\subset X\times Y$ with $(\mu\times \nu)(G)\ge 1-4\epsilon$, we can find an $x\in X$ and an $r$-tube $T$ containing $x$ such that
\[
\nu(T\cap G|_x) \ge K'\cdot r^{\sigma + \eta}.
\]
The first thing that the authors do is to choose a suitable $G$.

Because $(\mu,\nu)$ has $\sigma$-thin tubes, we start with the $G$ given to us by the definition. For each $x\in X$, 
\[
\nu(T\cap G|_x) \le K\cdot r^\sigma\ \quad \text{for all $r > 0$ and all $r$-tubes $T$ containing $x$},
\]
and moreover, $(\mu\times \nu)(G) \ge 1-\epsilon$.

We introduce the \emph{bad} dyadic $r$-tubes containing $x$; $\mathcal T_{x,r}''$ is the collection of $r$-tubes $T$ containing $x$ such that $\nu(T\cap G|_x) \ge \frac{K'}{2}\cdot r^{\sigma+\eta}$, for an appropriate $K'$ we will have to choose later.

As in the proof of the Key Lemma for the Kaufman-type radial projection theorem, we need to pare down the potentially infinite collection $\mathcal T_{x,r}''$. For each $x$, we choose a maximal subset of $\mathcal M\subset \mathcal T_{x,r}''$ with $r$-separated angles, and then let $\mathcal T_{x,r}' = \{10T:T\in\mathcal M\}$ be the collection of $10r$-tubes we obtain by dilating each member of $\mathcal M$ by a factor of $10$. By maximality and inflating, it ensures that $(\cup\mathcal T_{x,r}'')\cap B^2 \subset \cup\mathcal T_{x,r}'$.

We also verify the cardinality estimate $|\mathcal T_{x,r}'| \lesssim (K')^{-1}\cdot r^{-\sigma-\eta}$. Let $C$ denote the unit circle centered on $x$. Because members of $\mathcal M$ have $r$-separated directions, when we intersect members of $\mathcal M$ with $C$, the resulting family of caps on $C$ is at most finitely overlapping, so we have
\begin{align*}
    1 = \pi_x\nu(C) &\ge \pi_x\nu(C\cap (\cup\mathcal M))\\
    &\sim \sum_{T\in\mathcal M}\pi_x\nu(C\cap T) \\
    &\sim \sum_{T\in\mathcal M}\nu(T) \gtrsim (K'\cdot  r^{\sigma+\eta})|\mathcal T_{x,r}'|.
\end{align*}
In the last line, we used that $|\mathcal M| = |\mathcal T_{x,r}'|$, by definition. The next point is to ensure that our set $G$ we pick satisfies the property that whenever $x\in X$, we have $(x,y)\in G$ if and only if $y$ is covered by at least one bad tube centered on $x$.

We make some general observations at this point to motivate the next step of the argument. If $H\subset X\times Y$ is \emph{any} set that has very small $\mu\times \nu$-measure, say $(\mu\times\nu)(H) < 4\epsilon$, then $(\mu\times\nu)(H^c)\ge 1-4\epsilon$. Because $(\mu,\nu)$ does \emph{not} have $(\sigma+\eta,K',1-4\epsilon)$-thin tubes by assumption, there exists $x\in H^c$ and an $r$-tube $T$ containing $x$ such that
\[
\nu(T\cap (H^c)|_x)\ge \frac{K'}{2}\cdot r^{\sigma+\eta}>0.
\]
In particular, $\nu(T\cap (H^c)|_x) > 0$, so there must be a point $y\in T\cap (H^c)|_x$. In other words, if $H$ is too small by $\mu\times\nu$-measure, there must be an $x\in X$ such that it is impossible to cover all of $H|_x$ by tubes $T$ containing $x$.

Returning to the authors' argument, for fixed dyadic $r$, we let $\bar H_r = \{(x,y)\in X\times Y:y\in\cup\mathcal T_{x,r}'\}$, and let $\bar H$ be the grand union $\bigcup_{r}\bar H_r$. By the general remarks we have just made, we know that $(\mu\times\nu)(\bar H)\ge2\epsilon$, since for each point $(x,y)\in\bar H$, we \emph{do} have $y\in  T\cap (\bar H|_x)$ for some $r$-tube $T$, by the definition of $\bar H$. We even have room to intersect $\bar H$ with our original $G$, so we have $H := G\cap \bar H$, and $H_r = G\cap \bar H_r$, with $(\mu\times\nu)(H)\ge \epsilon$ by inclusion-exclusion. 

By the union bound,
\begin{equation}\label{eq:union-bound}
(\mu\times \nu)(H) \le \sum_{r}(\mu\times\nu)(H_r)
\end{equation}
However, $K'$ is a free parameter, so for large values of $r$ the set $H_r = \varnothing$. This is because for any fixed $r_0$, we can choose $K'=K'(r_0,\sigma,\eta)$ to be so large that $\frac{K'}{2}\cdot r_0^{\sigma+\eta}>1$ holds, and hence the defining relation for $\mathcal T_{x,r_0}''$,
\[
\nu(T\cap G|_x) \ge \frac{K'}{2}\cdot r_0^{\sigma+\eta},
\]
is impossible to satisfy. Thus, fix $r_0$ sufficiently small so $\sum_{r\le r_0}r^\eta < \epsilon/2$, and then choose $K'$ so large that $(K'/2)\cdot r_0^{\sigma+\eta} > 1$, so we can update inequality \eqref{eq:union-bound} to say
\begin{equation}\label{eq:updated-union-bd}
    \epsilon\le (\mu\times\nu)(H)\le\sum_{r\le r_0}(\mu\times\nu)(H_r).
\end{equation}
By the first inequality of \eqref{eq:updated-union-bd} and the pigeonhole principle, there must be a $\ar\le r_0$ such that $(\mu\times\nu)(H_\ar)\ge 2\ar^\eta$. The set $H_\ar$ has the property that all of its $x$-sections are covered by bad $\ar$-tubes, and this ``bad scale" $\ar$ is now fixed.

We set $\mathbf X = \{x\in X:\nu(H_\ar|_x)\ge \ar^\eta\}$, which is almost the final set of ``base points'' for a discrete Furstenberg set; the only thing left to do is appeal to the version of Frostman's lemma which allows us to find an $(\ar,s,\ar^{-O(\eta)})$-set $P_X\subset\mathbf X$. To finish building the Furstenberg setup to invoke Theorem \ref{thm:fu-ren}, we need to choose the tube families $\mathcal T_x$, and prove that $\mathcal T_x$ is an $(\ar,\sigma,\ar^{-O(\eta)})$-set of tubes. What the authors do first is to pick the heavy tubes from the collection $\mathcal T_{x,\ar}'$:
\[
\mathcal T_x = \{T\in\mathcal T_{x,\ar}':\nu(T\cap H_\ar|_x) \ge \ar^{\sigma+3\eta}\}.
\]
Since $H_\ar|_x\subset\cup\mathcal T_{x,\ar}'$, we can write
\[
H_\ar|_x = \bigcup_{T\in\mathcal T_x}T\cap H_\ar|_x + \bigcup_{T\in\mathcal T_{x,\ar}'\setminus \mathcal T_x}T\cap H_\ar|_x.
\]
There are no more than $(K')^{-1}\cdot \ar^{-\sigma-\eta}$ tubes in $\mathcal T_{x,\ar}'\setminus\mathcal T_x$, each of which is a ``light'' tube with $\nu(T\cap H_\ar|_x)< r^{\sigma+3\eta}$, so with $Y_x := (H_\ar|_x)\cap(\cup\mathcal T_x)$,
\[
\ar^\eta \le \nu(H_\ar|_x) \le \nu(Y_x) + (K')^{-1}\ar^{2\eta}.
\]
The second term can be absorbed into the left-hand side, so we know that $\nu(Y_x) \ge \ar^{2\eta}$. In other words, by just keeping the heavy tubes and replacing $H_\ar|_x$ by $Y_x$, we don't throw away much $\nu$-mass. Now we use our assumption that $(\mu,\nu)$ has $\sigma$-thin tubes: for each $T\in\mathcal T_x$, and each $\rho\ge \ar$, letting $T^{(\rho)}$ denote the $\rho$-neighborhood of $T$, we have
\[
\nu(T^{(\rho)}\cap Y_x) \le \nu(T^{(\rho)}\cap G|_x) \lesssim \rho^\sigma.
\]
In particular, because the tubes in $\mathcal T_x$ are heavy and have $\ar$-separated directions, we have
\begin{align*}
\ar^{\sigma+3\eta}|\{T'\in\mathcal T_x:T'\subset T^{(\rho)}\}|& \le \sum_{T'\in\mathcal T_x:T'\subset T^{(\rho)}}\nu(T'\cap Y_x) \\[0.1cm]
&\le \nu(T^{(\rho)}\cap Y_x) \lesssim \rho^\sigma.
\end{align*}
Therefore, $|\{T'\in\mathcal T_x:T'\subset T^{(\rho)}\}|\lesssim \ar^{-3\eta}(\frac{\rho}{\ar})^\sigma$. Lastly, $\ar^{2\eta}\le \nu(Y_x) \le \sum_{T\in\mathcal T_x}\nu(T\cap Y_x) \le \ar^\sigma|\mathcal T_x|$, which implies $|\mathcal T_x| \ge \ar^{-\sigma+2\eta}$. Since all the tubes $T\in\mathcal T_x$ pass through $x$, cardinality is comparable to $\ar$-covering number, so if $\ell$ is a line and $B(\ell,\rho)$ denotes the $\rho$-neighborhood of $\ell$ in $\mathcal A(2,1)$, we have
\[
|\mathcal T_x\cap B(\ell,\rho)|_\ar \lesssim \ar^{-O(\eta)}\rho^\sigma\cdot|\mathcal T_x|_\ar.
\]
This verifies that $\mathcal T_x$ is an $(\ar,\sigma,\ar^{-O(\eta)})$-set.

\begin{rem}
We note that to produce the set $\mathbf X$ and the tube families $\mathcal T_x$ for $x\in\mathbf X$ we used both assumptions that $(\mu,\nu)$ has $\sigma$-thin tubes, and $(\mu,\nu)$ does not have $(\sigma+\eta)$-thin tubes, supporting what we said in our zoomed out look at the Key Lemma regarding the connection between this counter-assumption and a dual Furstenberg setup.
\end{rem}

To finish the setup for the invocation of the Fu--Ren incidence theorem in the form of Theorem \ref{thm:fu-ren}, we need to choose the set $P_Y$. Instead of taking an arbitrary $(\ar,t,\ar^{-O(\eta)})$-set in $Y$, which exists by Frostman's lemma, we need to find a set for which we have a good upper bound of its cardinality. We say a few more words here about the construction of the set $P_Y$. For each $j\ge 0$, we let
\[
Y_j  = \{y\in Y: 2^{-j-1}\cdot C\ar^t < \nu(B(y,\ar)) \le 2^{-j}\cdot C\ar^t\}.
\]
First we explain what the authors mean by there only being $O(\log(1/\ar))$ many choices of ``$j$'' which need to be considered. For each $j$, let $\mathcal B_j = \{B(y,\frac{\ar}{5}):y\in Y_j\}$ be a covering of $Y_j$ by centered balls of radius $\ar/5$. By a Vitali covering argument, we can extract a disjoint subcollection $\mathcal B_j' \subset\mathcal B_j$ such that $Y_j\subset \bigcup \mathcal B'_j$. By definition of $Y_j$,
\[
\nu(Y_j) \le \sum_{\mathcal B'_j} \nu(B(y,\ar)) \le 2^{-j}\cdot C\ar^t \cdot \#\mathcal B_j'.
\]
As the balls in $\mathcal B_j'$ are disjoint and contained in $B(0,C)$ by the compact support of $\nu$, 
 $\#\mathcal B_j' \lesssim \ar^{-2}$. Therefore, if $j > A\log(1/\ar)$ for a large enough constant $A$, we have $\nu(Y_j) \lesssim 2^{-j}\ar^{t-2} \lesssim 2^{-j/2}\ar^{5\eta}$. Consequently, for each $x$, we can write
\begin{align}
\ar^{2\eta} \le \nu(Y_x) &\lesssim \sum_{j\le A\log(1/\ar)}\nu(Y_x\cap Y_j) + \sum_{j>A\log(1/\ar)}2^{-j/2}\ar^{5\eta} \\
&\le \sum_{j\le A\log(1/\ar)}\nu(Y_x\cap Y_j) + \ar^{4\eta},
\end{align}
provided we chose $A$ sufficiently large. Absorbing $\ar^{4\eta}$ into the left-hand side, we see for each $x$, there exists $j(x)\le A\log(1/\ar)$ such that $\nu(Y_x\cap Y_j)\ge A^{-1}\log(1/\ar)^{-1}(\ar^{2\eta}-\ar^{4\eta})\ge \ar^{3\eta}$, provided we choose $r_0$ small enough depending on $A$.

Next, for each $0\le j\le A\log(1/\ar)$, we let $P_X(j) = \{x\in P_X:j(x) = j\}$. As there are only $O(\log(1/\ar))$-many values of $j$, by the pigeonhole principle, there is a particular $j$ so that with $P_X':= P_X(j)$, we have $|P_X'| \ge \ar^\eta|P_X|$. We let $P_Y$ be a maximal $\ar$-separated subset of the set $Y_j$, which by its construction has cardinality upper bounded by $C^{-1}\cdot 2^j\ar^{-t}$. As the authors verify, $P_Y$ is an $(\ar,t,\ar^{-O(\eta)})$-set.

Since $\nu(Y_x\cap Y_j) \ge \ar^{3\eta}$, by considering which tubes in $\mathcal T_x$ are heavy and light for the set $Y_x\cap Y_j$, we find a subset $\mathcal T_x'\subset\mathcal T_x$ such that $\nu(Y_x\cap Y_j\cap T) \ge \ar^{\sigma+O(\eta)}$ for all $T\in\mathcal T_x'$, and such that $|\mathcal T_x'|\ge \ar^{O(\eta)}|\mathcal T_x|$.  For each $T\in\mathcal T_x'$, by what's been shown, we have
\[
\ar^{\sigma+O(\eta)}\le \nu(Y_x\cap Y_j\cap T) \lesssim (2^{-j}\cdot C\ar^t)\cdot|P_Y\cap 2T|,
\]
the last upper bound following by covering $Y_x\cap Y_j\cap T$ by $\ar$-balls and using the definition of $Y_j$. Since we have the good upper bound on $|P_Y|\lesssim 2^j\cdot C\ar^{-t}$, this string of inequalities translates into the lower bound:
\[
|P_Y\cap 2T| \gtrsim \ar^{\sigma+O(\eta)}|P_Y|.
\]
With this, we set $\mathcal T_x'' = \{2T:T\in\mathcal T_x'\}$ to finish the construction of the tube families we need to invoke Theorem \ref{thm:fu-ren}. This finishes the proof of Lemma \ref{lem:falconer epsilon improvement}.




\section{Further results and current work}\label{s:Section 4}

\subsection{Strengthening Theorem 4.2.ii}\label{ss:Section 4.1}

Section 4 of OSW generalizes Theorem 1.1 to higher dimensions in Theorem 4.2.ii:

\begin{thm}[\cite{orponen2022kaufman} Theorem 4.2.ii] \label{thm:thm4.2.ii}
    If $X\subset \R^n$ be Borel with $\dim X \in (k-1,k]$ (for some $k \in \{1,\dots, n-1\}$) such that $X$ is not contained in any $k$-plane, the following holds. If $\dim Y > k - 1/k - \eta$ for a sufficiently small constant $\eta = \eta(n,k,\dim X)>0$, then 
\[
\sup_{x\in X} \dim \pi_x(Y\setminus \{x\}) \geq \min \{\dim X, \dim Y\}.
\]
For $k=1$, we require no lower bound from $\dim Y$.
\end{thm}

Here, notably, the lower bound of $\dim Y> k-1/k - \eta$ was conjectured to not be the correct lower bound; rather, OSW conjectured that the lower bound can be replaced by $\dim Y>k-1$. This conjecture has been resolved in recent work of Kevin Ren \cite{ren2023discretized}. In particular, he showed the following:
\begin{thm}[\cite{ren2023discretized} Theorem 1.1]
    Let $X,Y\subset \R^n$ be Borel sets with $\dim X, \dim Y\leq k$ for some integer $1\leq k\leq n$. If $X$ is not contained in a $k$-plane, then 
    \[
    \sup_{x\in X} \dim \pi_x(Y\setminus \{x\}) \geq \min \{\dim X, \dim Y\}.
    \]
\end{thm}

In the proof of this theorem, Ren first strengthens the planar $\epsilon$-improved Furstenberg set estimate to higher dimensions (\cite{ren2023discretized} Theorem 1.3). After proving this result, he uses a bootstrapping method much like that in \S2 of OSW to directly prove the 1-codimensional cases (Theorem \ref{thm:thm4.2.ii} above) with the conjectured lower bound. Proving this for the 1-codimensional cases completes the proof for all other values of $k$ by the same reduction done in OSW Proposition 4.3.

Ren's result in \cite{ren2023discretized} was further utilized in two other recent papers: one on the Falconer distance set problem \cite{du2023new} and one on weighted refined decoupling estimates \cite{du2023weighted}, both by Du--Ou--Ren--Zhang.

\subsection{Generalizing Theorem 1.2} \label{ss:Section 4.2}

In the above subsection, we briefly discussed current work being done completing a higher dimensional Theorem 1.1. However, in OSW, not much was done in the way of generalizing Theorem 1.2. 

To some extent, this makes sense for, as we saw in the proof of Theorem 1.2, the key ingredient was the Fu--Ren \textit{planar} incidence estimate between $\delta$-balls and $\delta$-tubes. However, if this incidence estimate were to be generalized to $\R^n$, one might expect the framework of OSW to carry over (due to the generality of Lemma \ref{lem:KEY thin tubes}). This is precisely done in ongoing work of the first author, Fu, and Ren. Furthermore, in this paper, the first author, Fu, and Ren exhibit a shorter proof of OSW's Theorem 1.2 utilizing Theorem \ref{thm:bright thm1}.

\subsection{From point-line geometry to radial projections}
In the course of studying \cite{orponen2022kaufman}, we focused on many of the geometric parallels between radial projection of fractal sets $X \subset \mathbb{R}^2$ and the incidence of point sets $P \subset \mathbb{R}^2$ and their associated set of lines $\mathcal{L}(P)$. While Beck's theorem (see \S\ref{ss:becks theorem}) was explicitly mentioned by the authors of \cite{orponen2022kaufman}, we found the following two results from point-line geometry to be especially interesting.

\begin{thm}[\cite{beck1983lattice} Weak Dirac conjecture]\label{diracconj}
    Suppose that $P \subset \mathbb{R}^2$ is a finite point set, with $|P| = N$, and that not all points of $P$ are collinear. Let $\mathcal{L}(P)$ denote the set of connecting lines for $P$. Then,
    \begin{equation}\label{dirac}
        \max_{p\in P}|\{\ell\in\mathcal L(P):p\in \ell\}|\gtrsim N.
    \end{equation}
\end{thm}

The special case of Theorem 1.1 in \cite{orponen2022kaufman} where $X=Y$ can be viewed as a continuum analog of the Weak Dirac conjecture, because the cardinality of the set of lines in $\mathcal L(P)$ containing $p$, which equals $|\pi_p(P\setminus\{p\})|$, can be viewed as the analog of the dimension of $\pi_x(X\setminus\{x\})$ in the continuum setting.






The next theorem is a refinement of Beck's theorem that takes into account the number of points which lie on a given line of $\mathcal L(P)$.

\begin{thm}[\cite{beck1983lattice} Erd\H{o}s--Beck theorem]\label{erdosthm}
    Let $0\le k \le N-2$. For a finite set $P \subset \mathbb{R}^2$, with $|P| = N$, let $\mathcal{L}(P)$ denote the set of all lines spanned by pairs of points in $P$. If $\max_{\ell \in \mathcal{L}(P)} |P\cap \ell| = N - k$, then $| \mathcal{L}(P)| \gtrsim N \cdot k$, with a universal implied constant.
\end{thm}




The method used by OSW to prove Beck's theorem also provides partial progress towards a continuum analog of the Erd\H{o}s--Beck theorem.

\begin{thm}[Towards a continuous Erd\H{o}s--Beck theorem]\label{thm:progressEB}
    Let $X \subset \mathbb{R}^2$ be a Borel set. Suppose there exists some $0 \leq t \leq \dim X$ such that $\dim (X \setminus \ell) \geq \dim X - t$ for every line $\ell \subset \mathbb{R}^2$. Then the set $\mathcal{L}(X)$ of connected lines for $X$ satisfies
    \begin{equation}
        \dim \mathcal{L}(X) \geq \min \{2 \dim X - 2t, 2\}.
    \end{equation}
\end{thm}
    \textit{Proof.} This follows from the same methodology as the proof of Corollary 1.4 in \cite{orponen2022kaufman}. The main difference being that we consider the set
    \begin{equation*}
        B_t : =\{x \in X : \dim \pi_x (X \setminus \{x\}) < \min  \{\dim X - t, 1\} \},
    \end{equation*}
    and then demonstrate as before that
    \begin{equation*}
        \dim (X \setminus B_t) = \dim X.
    \end{equation*}
    This argument is essentially unchanged from that of OSW in \cite{orponen2022kaufman}. 
    
    For each $x \in X \setminus B_t$, we identify $\pi_x(X\setminus \{x\})$ with the collection of lines through $x$, denoted $\mathcal{L}_x$. Note that as $\dim \pi_x(X\setminus \{x\})\geq \min \{\dim X -t, 1\}$, we have that $\dim \mathcal{L}_x \geq \min \{\dim X - t, 1 \} := s$. We then have 
    \begin{equation*}
        \mathcal{L}(X) \supset \bigcup_{x \in X \setminus B_t} \mathcal{L}_x,
    \end{equation*}
    and $\bigcup_{x\in X\setminus B_t} \mathcal{L}_x$ contains a $(s,\dim X)$-Furstenberg set of lines $F$ (see Definition \ref{def:furst-lines}). As such, by the classical bound $\dim F \ge s + \min\{s,\dim X\} = 2s$, we have
    \begin{equation*}
        \dim \mathcal{L}(X) \geq \min \{2\dim X - 2t, 2\}. \mathqed
    \end{equation*}

Comparing Theorem \ref{erdosthm} and Theorem \ref{thm:progressEB}, one notices a distinction in numerology. In particular, from the discrete Erd\H{o}s--Beck theorem, one might expect a lower bound in the continuum setting akin to ``$\dim X + t$'', as this is akin to the logarithm of ``$N\cdot k$'' from the discrete setting. In fact, such a lower bound is currently being explored by the first and third authors, utilizing key results from both the work of Kevin Ren discussed in \S\ref{ss:Section 4.1}, and the work of Fu, Ren, and the first author discussed in \S\ref{ss:Section 4.2}. In fact, the first and third authors prove the result in all dimensions for unions of lines.

\appendix
\renewcommand{\theprop}{\Alph{section}.\arabic{prop}\rmfamily}

\section[Proving Beck's theorem using an epsilon-improvement]{Proving Beck's theorem using an \for{toc}{\texorpdfstring{$\bm{\epsilon}$}{epsilon}}\except{toc}{\texorpdfstring{$\epsilon$}{epsilon}}-improvement} \label{s:Beck}

Let $P \subset \R^2$ be a finite set of points, and let $\mathcal{L}$ be a finite set of lines in the plane. Furthermore, let 
\[
I(P,\mathcal{L}) := \{(p,\ell) \in P \times \mathcal L : p \in \ell\}
\]
be the set of \textit{incidences} of $P$ and $\mathcal{L}$.

\subsection[Quantifying an appropriate epsilon-improvement]{Quantifying an appropriate \for{toc}{\texorpdfstring{$\epsilon$}{epsilon}}\except{toc}{\texorpdfstring{$\bm{\epsilon}$}{epsilon}}-improvement}
An elementary argument utilizing the Cauchy--Schwarz inequality yields,
\begin{equation}\label{eq:cauchy--schwarz}
|I(P,\mathcal{L})| \leq |P| + |\mathcal{L}| + (|P||\mathcal{L}|)^{\frac{3}{4}}.
\end{equation}

To prove \eqref{eq:cauchy--schwarz}, by the Cauchy--Schwarz inequality,
\begin{align*}
    |I(P,\mathcal L)|^2 &= \left( \sum_{p\in P}\sum_{\ell\in\mathcal L}1_{p\in \ell} \right)^{\!2} \\
    &\le |P|\sum_{p\in P}\sum_{\ell_1,\ell_2\in\mathcal L}1_{p\in\ell_1}1_{p\in\ell_2} \\
    &= |P| \left( |I(P,\mathcal L)| + \sum_{\ell_1\ne\ell_2}\sum_{p\in P}1_{p\in\ell_1\cap \ell_2} \right).
\end{align*}
Finally, for each pair of distinct lines $\ell_1\ne\ell_2$, there is at most one $p\in P$ contained in $\ell_1\cap \ell_2$, so we have
\[
|I(P,\mathcal L)|^2 \le |P|(|I(P,\mathcal L)|+|\mathcal L|^2),
\]
from which we deduce
\begin{equation*}
    |I (P, \mathcal{L})| \leq |\mathcal L| + |\mathcal L|^{1/2} |P|.
\end{equation*}
Interchanging the roles of points and lines produces the analogous inequality,
\begin{equation*}
|I (P, \mathcal{L})| \leq |P| + |P|^{1/2} |\mathcal L|.
\end{equation*}
Taking the geometric mean of these two inequalities produces \eqref{eq:cauchy--schwarz}. 

To prove Beck's theorem, we will suppose that we have an $\epsilon$-improvement of the estimate in \eqref{eq:cauchy--schwarz} of the following form. Suppose for some $\epsilon \in (0,1/6)$, we knew the incidence estimate
\begin{equation} \label{eq:cauchy--schwarz-epsilon}
    |I(P,\mathcal{L})| \lesssim m + n + m^{1/2 + \epsilon} n^{1 - 2 \epsilon}
\end{equation}
holds for every set of $n$ points and $m$ lines. The case $\epsilon = 0$ is the Cauchy--Schwarz estimate, and $\epsilon = 1/6$ is the Szemer\'edi--Trotter estimate, so \eqref{eq:cauchy--schwarz-epsilon} is an $\epsilon$-improvement of Cauchy--Schwarz in an interpolation sense.



\subsection{The proof of Beck's theorem} We now present a proof of Beck's theorem using the weaker ``$\epsilon$-improvement" incidence estimate of \eqref{eq:cauchy--schwarz-epsilon}. Recall that $P$ is a set of $n$ points and $\mathcal L(P)$ denotes the lines spanned by $P$. Further, we will let $Q := P \times P \setminus \{(x,x):x\in P\}$.

For $r\ge 2$ a dyadic number, let
\begin{equation*}
    \mathcal L_r := \{ \ell \in \mathcal L : |\ell \cap P| \geq r \} 
\end{equation*}
be the collection of \emph{$r$-rich lines} of $\mathcal L$ and
\begin{equation*}
    T_r := \{ (x,y) \in Q  : \exists \, \ell \in \mathcal L_r\setminus\mathcal L_{2r}\ \text{such that}\  x,y \in \ell\}.
\end{equation*}
be the collection of \textit{$r$-connected pairs} of points in $P$. By definition,
\[
|\mathcal L(P)| \sim \sum_{2\le r\le n}|T_r|.
\]
 For each pair $(x,y) \in T_r$, there exists a unique $\ell \in \mathcal L_r$ with $x,y \in \ell$. Each $\ell \in \mathcal L_r$ contains at least $r$ points of $X$, so
\begin{equation*}
    r^2 |\mathcal L_r| \lesssim \binom{n}{2} \sim n^2, \quad \text{i.e.,} \quad |\mathcal L_r| \lesssim \frac{n^2}{r^2}.
\end{equation*}
Plugging this into \eqref{eq:cauchy--schwarz-epsilon} with $\mathcal L_r$ in place of $\mathcal L$ yields, for each $r\ge 2$,
\begin{equation*} \label{eq:upper-incidence-bound}
    |I(P,\mathcal L_r)| \lesssim \frac{n^2}{r^2} + n + \left (\frac{n^2}{r^2} \right)^{\!1/2 + \epsilon} n^{1 - 2 \epsilon} \sim n  + \frac{n^2}{r^{1+2\epsilon}}.
\end{equation*}
In particular,
\begin{equation}\label{eq:richnessest}
    |I(P,\mathcal L_r)| \lesssim \frac{n^2}{r^{1+2\epsilon}}+n, \textrm{ for each } r\ge 2. 
\end{equation}

On the other hand, each line in $\mathcal L_r$ gives rise to at least $r$ incidences (and, of course, no two lines give rise to the same incidence), so 
\begin{equation} \label{eq:lower-incidence-bound}
    r |\mathcal L_r| \leq |I(P,\mathcal L_r)|.
\end{equation}
To estimate the cardinality of the set $\mathcal L(P)$ of lines spanned by $P$, we must estimate each $|T_r|$ and sum over $r$. The map $T_r \to \mathcal L_r$ taking a pair of points to the line containing them is (at least) $r^2$-to-$1$. Combining \eqref{eq:richnessest} with \eqref{eq:lower-incidence-bound} gives,
\begin{equation} \label{eq:j-connected-bound}
    |T_r| \leq r^2 |\mathcal L_r| \lesssim \frac{n^2}{r^{2\epsilon}}+nr.
\end{equation}

Let $C$ be a large constant (to be determined shortly). Summing \eqref{eq:j-connected-bound} over all dyadic $C < r < n/C$ gives:
\begin{align*}
    \sum_{C < r < n/C} |T_r| &\lesssim n^2\sum_{C < r < n/C} \frac{1}{r^{2\epsilon}} + n\sum_{C < r < n/C}r\\ 
    &\lesssim n^2 \left( \frac{1}{C^{2\epsilon}} + \frac{1}{C} \right) \sim \frac{n^2}{C^{2\epsilon}}
\end{align*}

Now, recall that $|T_r|$ was the number of $r$-connected pairs in $Q : = P \times P \setminus \diag (P)$.  So, let $T = \bigcup_{C\leq r \leq n/C} T_r$. Since the $T_r$ are pairwise-disjoint, we have that $|T| \lesssim \frac{n^2}{C^{2 \epsilon}}$. However, the total number of pairs of points forming lines in $\mathcal L$ is $|Q| := \frac{n^2-n}{2}$. Hence, by choosing $C$ large enough depending upon $\epsilon > 0$ (but independent of $n$) we find a set $P'\subset P$ of size $\gtrsim n^2$ that is not $r$-connected for any $C < r < n/C$.

Let $\mathcal L(P')\subset \mathcal{L}(P)$ denote the lines determined by the pairs of points in $P'$. These lines either pass through fewer than $C$-many points, or otherwise pass through more than $n/C$ many-points. The latter satisfies \textbf{(1)} of Theorem \ref{becktheorem}.

For the former, we may assume that all lines of $\mathcal L(P')$ connect $\leq C$-many points. Yet, each line can connect at most $\sim C^2$-many pairs of points from $P'$. Therefore, we see that there must by at least $\frac{n^2}{C^2}$-many lines formed by the pairs of points in $P$. This satisfies \textbf{(2)} of Theorem \ref{becktheorem}, and it concludes the proof.

\section{Tools from geometric measure theory} \label{Appendix B}

A general-purpose tool we use for proving lower bounds for the Hausdorff dimension of sets is the following. Recall that a finite nonzero Borel measure $\mu$ on a metric space $\Omega$ is \define{$\bm{s}$-Frostman} if there exists a constant $C > 0$ such that
\begin{equation*}
    \mu(B(x,r)) \leq C \+ r^s \qquad \forall \+ x \in \Omega \text{ and } \forall \+ r > 0.
\end{equation*}

\begin{lem}[Mass distribution principle] \label{lem:dim lower bound}
    If a Borel set $X$ supports an $s$-Frostman probability measure $\mu$, then $\mathcal H^s(X) > 0$.
\end{lem}
\begin{proof}
    For any cover of $X$ by balls $\{B_i\}_{i=1}^\infty$, by the Frostman condition on $\mu$, we have
    \[
    \sum_{i=1}^\infty r_i^s \ \gtrsim \ \sum_{i=1}^\infty \mu(B_i).
    \]
    By assumption, $\{B_i\}_{i=1}^\infty$ is a cover of $X$ and $X \supset \spt \mu$, so $1 = \mu(X) \le \sum_{i=1}^\infty \mu(B_i)$. As the cover was arbitrary, $\mathcal{H}^s(X) \gtrsim 1$.
\end{proof}
In particular, Lemma \ref{lem:dim lower bound} implies that a Borel set supporting an $s$-dimensional Hausdorff measure has dimension at least $s$.

Converse to the mass distribution is Frostman's lemma (\cite{mattila1995geometry} Theorems 8.8 and 8.17), but what we require is a discretized version due to F\"{a}ssler and Orponen, which we paraphrase here. Let $\mathcal{H}_\infty^s$ denote the \textit{$s$-dimensional Hausdorff content}. It is well-known that a set $A \subset \R^n$ satisfies $\mathcal{H}^s(A) > 0$ if and only if $\mathcal{H}_\infty^s(A) > 0$.

\begin{lem}[\cite{fassler2014restricted} Discrete Frostman's lemma in $\R^2$] \label{lem:discrete-frostman}
    If $A \subset \R^2$ satisfies $\mathcal{H}_\infty^s(A) > 0$, then, for all sufficiently small $\rho > 0$, there exists a $(\rho,s)$-set $B \subset A$ with $|B| \gtrsim \mathcal{H}_\infty^\sigma(A) \+ \rho^{-s}$.
\end{lem}

See \cite{fassler2014restricted} Proposition A.1, where the proof involves only routine manipulations of dyadic cubes.

\section{Relation to orthogonal projections} \label{Appendix C}

The Kaufman estimate for orthogonal projections is closely related to the main Theorem \ref{thm:mainthm1.1}. Recall that the Kaufman estimate in the plane states the following:
\begin{prop}\label{prop:kaufman}
    Suppose $Y\subset \R^2$ is Borel. Then, for $0\leq \sigma < \min \{\dim Y, 1\}$, define the exceptional set 
    \[
    E_\sigma(Y) := \{e \in \mathbb{S}^1 : \dim P_e(Y) < \sigma \}.
    \]
    Then we have 
    \begin{equation}\label{eqn:kaufman}
    \dim E_\sigma(Y) \leq \sigma.
    \end{equation}
\end{prop}

As is stated in OSW, this statement is formally equivalent to the following: if $\varnothing \neq X \subset \mathbb{S}^1$ and $Y\subset \R^2$ are Borel, then 
\begin{equation}\label{eqn:ortho}
\sup_{e\in X} \dim P_e(Y) \geq \min \{\dim X, \dim Y, 1\}.
\end{equation}
This formulation of Proposition \ref{prop:kaufman} is analogous to the way Theorem \ref{thm:mainthm1.1} is stated, with the key difference being here we are considering orthogonal projections, rather than radial projections.

In this space we provide a proof of the formal equivalence of these two statements.

\begin{proof}[Proof that \eqref{eqn:kaufman} is equivalent to \eqref{eqn:ortho}]
We start by proving that \eqref{eqn:kaufman} implies \eqref{eqn:ortho} under the stated hypotheses. Let $\varnothing \neq X \subset \mathbb{S}^1$ and $Y\subset \R^2$ be Borel. Furthermore, let $\sigma = \min \{\dim X, \dim Y, 1\}$. If $\sigma = 0$, this implies that 
\[
\sup_{e\in X} \dim P_e(Y) \geq \min\{\dim X, \dim Y, 1\} = 0.
\]
which is certainly true. Hence, we may assume that $\sigma >0$, and let $\epsilon \in (0,\sigma)$.

Now, given that $X\subset \mathbb{S}^1$, applying \eqref{eqn:kaufman}, we have 
\[
    \dim \{e\in X : \dim P_e(Y) < \sigma - \epsilon\} \leq \dim \{e\in \mathbb{S}^1 : \dim P_e(Y) < \sigma - \epsilon\} \leq \sigma - \epsilon.
\]
(This is a Borel set\textemdash in fact, a $G_{\delta\sigma}$ set\textemdash by \cite{kaufman1968hausdorff} Statement A.) By definition of $\sigma$, $\sigma - \epsilon < \dim X$. This implies that for all $\epsilon \in (0,\sigma)$, $$X\not \subset \{e\in \mathbb{S}^1 : \dim P_e(Y) < \sigma - \epsilon\}.$$ Therefore, for all $\epsilon$ in this range, we have that there exists an $e_\epsilon \in X$ such that $\dim P_{e_\epsilon}(Y) \geq \sigma - \epsilon$. Hence,
\[
\sup_{e\in X} \dim P_e(Y) \geq \min \{\dim X, \dim Y, 1\}
\]
as desired.

We now show that \eqref{eqn:ortho} implies \eqref{eqn:kaufman} under the stated assumptions. Let $Y\subset \R^2$ be Borel, and let $0 \le \sigma \le \min\{\dim Y, 1\}$. If $\sigma = 0$, \eqref{eqn:kaufman} is immediate. Hence, suppose $0 < \sigma \leq \min \{\dim Y, 1\}$, let $\epsilon \in (0,\sigma)$, and define 
\[
X_\epsilon = \{e \in \mathbb{S}^1 : \dim P_e(Y) \leq \sigma - \epsilon\}.
\]
Since $E_\sigma(Y) = \bigcup_{\epsilon>0}X_\epsilon$, it suffices to prove for each $\epsilon\in(0,\sigma)$ that $\dim X_\epsilon \leq \sigma$. Suppose for the sake of contradiction that $\dim X_\epsilon>\sigma$. Since $X_\epsilon\ne\varnothing$, by  \eqref{eqn:ortho}, we have
\begin{equation}\label{eq:formal}
\sup_{e\in X_\epsilon} \dim P_e(Y) \geq \min \{\dim X_\epsilon,\dim Y, 1\}.
\end{equation}
We divide the remainder of the proof into cases based on the value of the right-hand side of this inequality.

\begin{enumerate}[leftmargin=2cm]
    \item[\textsc{Case 1.}] 
    Suppose $\dim X_\epsilon \leq \min \{\dim Y, 1\}.$ By \eqref{eq:formal}, $\sup_{e\in X_\epsilon}\dim P_e(Y) \ge \dim X_\epsilon$. Furthermore, for each $e\in X_\epsilon$, $\dim P_e(Y) \le \sigma-\epsilon$, so by our counter-assumption $\dim X_\epsilon > \sigma$ we have
    \[
    \sigma < \dim X_\epsilon \leq \sup_{e\in X_\epsilon} \dim P_e(Y) \leq \sigma - \epsilon.
    \]
    This is a contradiction.
    \item[\textsc{Case 2.}] Now suppose $\dim Y \leq \min \{\dim X_\epsilon, 1\}$. By the definition of the supremum, we have that for $0 < \epsilon_1 \leq \epsilon$, there exists $e_{\epsilon_1} \in X_\epsilon$ such that 
    \begin{align*}
        \dim Y - \epsilon_1 &< \dim P_{e_{\epsilon_1}}(Y) \\
        &\leq \sigma - \epsilon \\
        &< \dim Y - \epsilon.
    \end{align*}
    The last line is due to the fact that $\sigma < \min \{\dim Y, 1\}$. Therefore, we have that $\epsilon <\epsilon_1$, which is a contradiction.
    \item[\textsc{Case 3.}] Finally, consider the case when $1 \leq \min \{\dim X_\epsilon, \dim Y\}$. Then, by \eqref{eq:formal}, and the fact that $\sigma \leq 1$, we have that 
    \[
    1\leq \sup_{e\in X_\epsilon}\dim P_e(Y) \leq \sigma - \epsilon < 1.
    \]
    which is a contradiction.
\end{enumerate}
This finishes the proof that $\dim X_\epsilon \le \sigma$. As $\epsilon$ was arbitrary, it finishes the proof.
\end{proof}

\bibliographystyle{plain}
\bibliography{references}

\end{document}